\documentclass[preprint,12pt]{elsarticle}
\makeatletter 
\def\ps@pprintTitle{%
	\let\@oddhead\@empty
	\let\@evenhead\@empty
	\let\@oddfoot\@empty
	\let\@evenfoot\@oddfoot
}
\makeatother

\usepackage{bbm}
\usepackage{amsfonts}
\usepackage{amsmath}
\usepackage{amsthm}
\usepackage{datetime}
\usepackage{enumerate}
\usepackage{xcolor}
\usepackage{graphicx}
\usepackage{epstopdf}
\usepackage{natbib}
\setcitestyle{authoryear}
\newtheorem{thm}{Theorem}[section]
\newtheorem{lemma}{Lemma}[section]
\newtheorem{corollary}{Corollary}[section]
\newtheorem{prop}{Proposition}[section]
\newtheorem{remark}{Remark}[section]
\newtheorem{example}{Example}[section]

\newcommand{\dsum}{\displaystyle\sum}
\newcommand{\dint}{\displaystyle\int}

\newcommand{\cid}{\stackrel{d}{\longrightarrow}}
\newcommand{\cip}{\stackrel{P}{\longrightarrow}}
\newcommand{\cas}{\stackrel{a.s.}{\longrightarrow}}

\newcommand{\toi}{\to\infty}
\newcommand{\eind}{\stackrel{d}{=}}

    \newcommand{\1}{\mathbb{I}}
    
    \newcommand{\RR}{\mathbb{R}}
    
    \newcommand{\EE}{\mathbb{E}}
    \newcommand{\PP}{\mathbb{P}}

    \newcommand{\Exp}{\mathbb{E}}
    \newcommand{\Var}{\operatorname{Var}}
    
    \renewcommand{\Pr}{\mathbb{P}}
    \newcommand{\vep}{\varepsilon}
    \newcommand{\taui}{\Gamma_i}
    \newcommand{\taun}{\Gamma_n}

\begin{document}

\title{On total claim amount for marked Poisson cluster models} 

\author[a]{Bojan Basrak}
\ead{bbasrak@math.hr}
\address[a]{Department of Mathematics, University of Zagreb, Bijeni\v cka 30, Zagreb, Croatia}

\author[b]{Olivier Wintenberger}
\ead{olivier.wintenberger@upmc.fr}
\address[b]{LPSM, Sorbonne university, Jussieu, F-75005, Paris, France}

\author[c]{Petra \v Zugec\corref{cor1}}
\ead{petra.zugec@foi.hr}
\address[c]{Faculty of Organization and Informatics, Pavlinska 2, Vara\v zdin, University of Zagreb, Croatia}
\cortext[cor1]{Corresponding author}

\begin{abstract}
We study the asymptotic distribution of the total claim amount for marked Poisson cluster models. The marks determine the size and other characteristics of the individual claims and potentially influence arrival rate of the future claims. We find sufficient conditions under which the total  claim amount satisfies the central limit theorem or alternatively tends in distribution to an infinite variance stable random variable. We discuss several Poisson cluster models in detail, paying special attention to the  marked Hawkes processes as our key example.
\end{abstract}


\maketitle

\section{Introduction} 
Elegant mathematical analysis of the  classical Cram\'er--Lundberg risk model has a prominent place in nonlife insurance theory.
The theory yields precise or approximate computations of the ruin probabilities, appropriate reserves, distribution of the total claim amount and other properties of an idealized insurance portfolio,
see for instance \cite{AsmuAlbrecher} or \cite{mikosch}. In recent years, some special models have been proposed to account for  the possibility of clustering of insurance events. For instance,  in the context of Hawkes processes, some results on ruin probabilities can be found in \cite{stabile} and \cite{zhuIME2013}. General cluster point processes and Poisson cluster processes in particular, have been proved useful in a variety of fields when modelling events that cluster either in space or time. This includes seismology, telecommunications, forensic science, molecular biology or finance, we refer to section 6.4 in \cite{daley} for some examples.

The main goal of this article is to study asymptotic distribution of the total claim amount in the setting where Cram\'er--Lundberg risk model is augmented with a  Poisson cluster structure. To make this more precise,  we model arrival of claims in an insurance portfolio by a marked point process, say
$$
N = \dsum_{k=1}^\infty \delta_{\tau_k,A^k} \,,
$$
where $\tau_k$'s are nonnegative random variables representing arrival times with  some degree of clustering and $A^k$'s represent corresponding marks in a rather general metric space $\mathbb{S}$. Observe that we do allow for the possibility that marks influence arrival rate of the future claims. In the language of point processes theory,  we assume that the marks are merely unpredictable and not independent of the arrival times \citep{daley}. For each marked event, the claim size can be calculated using a measurable mapping of marks to nonnegative real numbers, $f(A^k)$ say. So that the total claim amount in the time interval $[0,t]$ can be calculated as  
\[
S(t) = \dsum_{\tau_k \leq t} f (A^k) = \dint_{[0,t]\times \mathbb{S}} f(a) N(ds, da)\,.
\]

In the sequel, we aim to determine the effect of the clustering  on the quantity $S(t)$, as $t\toi$ even in the case when the distribution of the individual claims does not satisfy assumptions of the classical central limit theorem. The paper is organized as follows --- in the following section we rigorously introduce marked Poisson cluster model and present some  specific cluster modes which have attracted attention in the related literature, see \cite{fayetal, stabile,zhu}. As a proposition in Section~\ref{sec:CLT} we present the central limit theorem for the total claim amount $S(t)$ in our setting under appropriate second moment conditions.  In Section \ref{sec:IVSL}, we prove a functional limit theorem  concerning the sums of regularly varying nonnegative random variables when subordinated to an independent renewal process. Based on this, we prove the limit theorem for the total claim amount $S(t)$ in cases when individual claims have infinite variance. Finally in Section~\ref{sec:APP} we apply our results to the models we introduced in Section~\ref{sec:GMEM}. In particular, we give a detailed analysis of the asymptotic behaviour of $S(t)$ for marked Hawkes processes which have been extensively studied in recent years.

\section{The general marked Poisson cluster model}\label{sec:GMEM}

  Consider an independently marked homogeneous Poisson point process with mean measure $(\nu$Leb) on the state space $[0,\infty)$ for some constant $\nu >0,$ where Leb denotes Lebesgue measure on  $[0,\infty),$  with marks in a completely metrizable  separable space
 $\mathbb{S}$, 
$$
N^0 = \dsum_{i\geq 1} \delta_{ \taui,A_i}\,.
$$
Marks $A_i$ are assumed to follow a common distribution $Q$ on a measurable
space $(\mathbb{S}, \mathcal{S})$
where $\mathcal{S}$ denotes a corresponding
Borel $\sigma$--algebra. 
In other words, $N^0$ is  a Poisson point process with intensity $\nu \times Q$ on the space
$[0,\infty) \times \mathbb{S}$\,.
For non--life insurance modelling purposes, the marks can take values in  $\RR^d$ with coordinates representing the size
of claim, type of claim, severity of accident, etc.

 Denote the space of locally finite point measures
on this space by $M_p= M_p ( [0,\infty) \times \mathbb{S})$ and assume 
that at each time $\taui$ with mark $A_i$ another
point process in $M_p $ is generated independently, we denote it by
$G^{A_i}$. Intuitively, point process $G^{A_i}$ represents a cluster of points that {is superimposed} on $N^0$ after time $ \taui$. Formally, there exists a probability kernel $K$, from $\mathbb{S}$ to $M_p$,
such that, conditionally on $N^0$, point processes $G^{A_i}$ are independent, a.s. finite and with the distribution equal to  $K(A_i,\cdot),$ thus the dependence between the $G^{A_i}$ and $A_i$ is permitted. Based on $N^0$ and clusters $G^{A_i}$ we define a cluster Poisson process.

In order to keep the track of the cluster structure, we can alternatively consider the process $G^{A_i}$ as a part of the mark attached to $N^0$ at time $ \taui$. Indeed,
$$
 \dsum_{i\geq 1} \delta_{ \taui,A_i,G^{A_i}}\,
$$
can be viewed as a marked Poisson process on  $[0,\infty)$ with marks in the space $  \mathbb{S} \times M_p$. We can write

$$
 G^{A_i} = \dsum_{j= 1}^{K_i}  \delta_{T_{ij},A_{ij}}\,,
$$ 
where ${(T_{ij})}_{j\geq 1}$ is a sequence of nonnegative random variables and for some $\mathbb{N}_0 $ valued random variable $K_i$. If we count the original   point arriving at time $\taui$, the actual cluster size is
$K_i+1$. Further, for any original arrival point {$\taui$} and corresponding
random cluster $G^{A_i}$, we introduce a  point process
$$
 C_i = \delta_{0,A_i} + G^{A_i}\,.
$$ Note that $K_i$ may possibly depend on $A_i$, but we do assume throughout
that 
\begin{equation} \label{Kfin}
\Exp K_i <\infty\,.
\end{equation}
Finally,  to describe 
 the size and other characteristics of the claims together with their  arrival times, we
use a marked point process   $N$ as a random element in $M_p$ of the
form
\begin{equation}\label{e:PoisProc}
N= \dsum_{i=1}^\infty  \dsum_{j = 0}^{K_i} \delta_{ \taui+T_{ij},A_{ij}}\,,
\end{equation}
where we set $T_{i0} = 0$ and $A_{i0} = A_i$. In this representation, the claims arriving at time $ \taui$ and corresponding to the index $j=0$ are called ancestral or  immigrant claims, while the claims arriving at times $ \taui+T_{ij},\ j \geq 1$, are referred to as progeny or offspring. Moreover, since $N$ is locally finite, one could also write
$$
N = \dsum_{k=1}^\infty \delta_{\tau_k,A^k} \,,
$$
with  $\tau_k\leq \tau_{k+1}$ for all $k \geq  1$. Note that in this representation we ignore the information regarding the clusters of the point process. Clearly, if the cluster processes $G^{A_i}$ are independently marked with the same mark distribution $Q$ independent of $A_i$, then all the marks $A^k$ are i.i.d. 

The size of claims
is  produced by an application of a measurable function, say 
 $f:\mathbb{S} \to \mathbb{R}_+$, on the marks. In particular,
 sum of all the claims due to the arrival of an immigrant claim at time $ \taui$
 equals
\begin{equation} \label{e:Di}
  D_i =\dint_{[0,\infty)\times \mathbb{S}} f(a) C_i(dt, da)\,,
\end{equation}
while 
  the total claim size in
the period $[0,t]$ can be calculated as 
$$
 S(t) = \dsum_{\tau_k \leq t} f (A^k) = \dint_{[0,t]\times \mathbb{S}} f(a) N(ds, da)\,.
$$

\begin{remark}
	In all our considerations, we take into account (without any real loss of generality)
	the original immigrant claims arriving at times $\taui$ as well. In principle,
	one could ignore these claims and treat  $\taui$  as times of incidents that trigger,
	with a possible delay, a cluster of subsequent payments. Such a  choice seems particularly  useful if one aims to model the so called incurred but not reported (IBNR) claims, when estimating appropriate reserves in an insurance portfolio \citep{mikosch}. In such a case, in the definition of the process $N$, one would omit the points of the original Poisson process $N^0$ and consider
$$
N= \dsum_{i=1}^\infty  \dsum_{j = 1}^{K_i} \delta_{ \taui+T_{ij},A_{ij}}\,,
$$	
instead.
\end{remark}
\subsection{Some special models} \label{ss:Mod}

Several examples of Poisson cluster processes have been studied in the monograph of \cite{daley}, see Example 6.3 therein for instance. Here we study  marked adaptation of the first three examples  6.3 (a)-(b) and (c) of \cite{daley}.

\subsubsection{Mixed binomial Poisson cluster process}\label{ss:modbin}

Assume that the clusters  have the following form
$$
  G^{A_i} = \dsum_{j=1}^{K_i} \delta_{W_{ij}, A_{ij}}\,,
$$
with $(K_i,(W_{ij})_{j\ge 1},(A_{ij})_{j\ge 0})_{i\ge 0}$ being an  i.i.d. sequence. Assume moreover that $(A_{ij})_{j\ge 0}$ are i.i.d. for any fixed $i=1,2,\ldots $ and that $(A_{ij})_{j\ge1}$ is independent of $K_i,(W_{ij})_{j\ge 1}$ for all $i\ge 0$. We allow for possible dependence between $K_i,(W_{ij})_{j\ge 1}$ and the ancestral mark $A_{i0}$, however, we  assume that $K_i$ and $(W_{ij})_{j\ge 1}$ are conditionally  independent given $A_{i0}$. As before we assume $\EE[K]<\infty$. Observe that we use notation $W_{ij}$ instead of $T_{ij}$ to emphasize relatively simple structure of clusters in this model in contrast with two other models in this section. Such a process $N$ is a version of the so--called Neyman--Scott  process, e.g. see {Example 6.3 (a) of \cite{daley}}.

\subsubsection{Renewal  Poisson cluster process}\label{ss:modren}

Assume next that the clusters  
$G^{A_i}$ 
have the following  distribution
$$
  G^{A_i} = \dsum_{j=1}^{K_i} \delta_{T_{ij}, A_{ij}}\,,
$$
where 
$(T_{ij})_j$ 
represents a renewal sequence
$$
T_{ij} = W_{i1} + \cdots +  W_{ij}\,,
$$
and we keep all the other assumptions from the model in subsection \ref{ss:modbin} (in particular,  $(W_{ij})_{j\ge 1}$ are conditionally i.i.d. and independent of $K_i$ given $A_{i0}$). 
 A general unmarked model of this type is called Bartlett–-Lewis
 model and analysed in \cite{daley}, see Example 6.3 (b).
 See also \cite{fayetal} for an application of such a point process to modelling of teletraffic data.

	 These  two simple cluster models  were already considered by \cite{mikosch} in the context of   insurance applications. In particular, subsection 11.3.2 and example 11.3.5 therein provide expressions for  the first two moments of the number of claims in a given time interval $[0,t]$.
	Both models
can be criticized as overly simple, still
 the assumption that claims (or delayed payouts) are separated by i.i.d. times (as in the renewal Poisson cluster process) often appears in the risk theory (cf. Sparre Andersen model, \cite{AsmuAlbrecher}).

%
\subsubsection{Marked Hawkes process}\label{ss:modhaw}

Key motivating example in our analysis is the so called (linear) marked Hawkes process. 
 Hawkes processes of this type have a neat Poisson cluster representation due to \cite{HaOa}. For this model,
 the clusters $G^{A}$ are recursive aggregation of Cox processes, i.e.
 Poisson processes with random mean measure $ \tilde{\mu}_A  \times Q$ where
 $  \tilde{\mu}_A $ has the following form
\begin{equation} \label{e:muA}
\tilde{\mu}_{A} (B) = \dint_B h(s, A) ds\,,
\end{equation}
for some  fertility (or self--exciting) function $h$, cf. {Example 6.4 (c) of \cite{daley}}. 

It is useful to
introduce a time shift operator $\theta_t$, by denoting  
$$
\theta _t  m =  \sum_j \delta_{t_j+t,a_j}\,,
$$ for an arbitrary point measure $m = \sum_j \delta_{t_j,a_j} \in M_p$ and $t\geq 0$.
Now, for the ground process $N^0 = \sum_{i\geq 1} \delta_{ \taui,A_i}\,$ which is a Poisson point process with intensity $\nu \times Q$ on the space
	$[0,\infty) \times \mathbb{S}$\,,
the cluster process corresponding to a point $( \Gamma,A)$
satisfies the following recursive relation
\begin{equation} \label{GA}
 G^A =  \dsum_{l=1}^{{L_A}}  \left(  {\delta_{\tau^1_l,A^1_l} +} \theta_{\tau^1_l} G^{A^1_l} \right)\,,
\end{equation}
where, given $A,$ $N^A = \sum_{l=1}^{{L_A}} \delta_{\tau^1_l,A^1_l}$ is a Poisson processes with random mean measure 
$ \tilde{\mu}_A  \times Q,$
the sequence $(G^{A^1_l})_l$ is i.i.d., distributed as $G^A$ and independent of $N^A$. 
Thus, at any ancestral point $(\Gamma,A)$  a cluster of points appears as a whole cascade of points to the right in time generated recursively according to \eqref{GA}.  Note that by definition $L_A$ has Poisson distribution conditionally on $A$, with
mean $\kappa_A=\int_0^\infty h(s,A) ds$. It corresponds to the number of the first generation progeny $(A^1_l)$ in the cascade. Note also that the point processes forming the second generation are again Poisson conditionally on the corresponding first generation  mark $A_l^1$. The cascade $G^A$ corresponds to the process formed by the successive generations, drawn recursively as Poisson processes given the former generation.

The marked Hawkes process is obtained by attaching to the ancestors $(\taui,A_{i})$ of the marked Poisson process $
N^0 = \sum_{i\geq 1} \delta_{ \taui,A_i}
$ a cluster of points, denoted by $C_i$, which contains point $(0,A_{i})$  and a whole cascade $G^{A_i}$ of points to the right in time generated recursively according to \eqref{GA} given $A_i$.
Under the assumption
\begin{equation} \label{e:kappa}
\kappa =  \Exp \dint h(s,A) ds < 1\,,
\end{equation}
the total number of points in a cluster is generated by a subcritical branching process. Therefore, 
the clusters are finite  almost surely, and we denote their size by
$K_i {+1}= C_i[0,\infty)$. It is known and not difficult to show that under \eqref{e:kappa},
the clusters always satisfy
$$
\Exp K_i {+1} = \frac{1}{1-\kappa}\,.
$$
Observe that the clusters $C_i$ are independent by construction and  can be represented as

\begin{equation}\label{e:CiHawkes}
C_i= \dsum_{j = 0}^{K_i} \delta_{\taui+T_{ij},A_{ij}}\,,
\end{equation}
with $A_{ij}$ being i.i.d. and $T_{i0} = 0$.
We note that in the case when marks do not influence conditional density, i.e. when $h(s,a) = h(s)$, random variable
$K_i{+1}$ has a so-called Borel distribution with parameter $\kappa$, see \cite{borel}.
Observe also that  in general, marks and arrival times of the final Hawkes process $N$  are not independent  of each other, rather, in the terminology
of \cite{daley}, the marks in the process $N$ are only unpredictable.

Hawkes processes are typically introduced through their conditional intensity. More precisely, a  point process $
N = \sum_{k} \delta_{\tau_k,A^k} \,,
$ represents a Hawkes process of this type if
the random marks $(A^k)$ are i.i.d. with distribution $Q$ on the space $\mathbb{S}$, while the arrivals $(\tau_k)$ have the
conditional intensity of the form 
\begin{equation}
\lambda_t = \lambda(t) = \nu + \sum_{\tau_i< t} h(t-\tau_i,A^i)\,,
\end{equation}
where $\nu > 0$ is a constant and  $h:[0,\infty)\times \mathbb{S}\to \mathbb{R}_+$ is assumed to be integrable  in the sense that
$\int_0^\infty \EE h(s,A) ds < \infty$.
Observe, $\nu$ is exactly the constant which determines the intensity of the underlying Poisson process $N^0$ due to the Poisson cluster representation of the linear Hawkes processes, cf. \cite{HaOa}.
Observe, $\lambda$ is
$\mathcal{F}_t$---predictable, where $\mathcal{F}_t$ stands for 
an internal history of $N$, 
$\mathcal{F}_t = \sigma\{N(I \times S): I \in \mathcal{B}(\mathbb{R}), I \subset (-\infty, t], S \in \mathcal{S}\}$. Moreover, $A^n$'s are assumed to be independent of the past arrival times  $\tau_i$, $i < n$, 
see also \cite{brem}.  
Writing $N_t = N((0,t] \times \mathbb{S})$, one can observe that $(N_t)$ is an integer valued process with nondecreasing paths.
The role of intensity can be described heuristically by the relation
$$
\Pr (dN_t =1 \mid \mathcal{F}_{t-}) \approx \lambda_t  dt\,.
$$

\subsubsection{Stationary version} \label{sss:STAT}

In any of the three examples above, the point process $N$
can be clearly made stationary if we start the construction  in \eqref{e:PoisProc}  on the state space
$\RR \times \mathbb{S}$ with a Poisson process $ \sum_i \delta_{\taui}$ on the whole real line. The resulting stationary cluster process is denoted by $N^*$. Still, from applied perspective, it seems more interesting 
to study the nonstationary version where both the ground process $N^0$ and the cluster process itself have arrivals only from some point onwards, e.g. in the interval $[0,\infty)$ as for instance in \cite{zhu}.

Stability of various cluster models, i.e. convergence towards a stationary distribution in appropriate sense has been extensively studied for various point processes. 
For instance, it is known that  the unmarked 
 Hawkes process on  $[0,\infty)$ converges 
to the stationary version on any compact set and on the positive line under the condition that 

\begin{equation}\label{eq:condstat}
\int_0^\infty s h (s) ds <\infty\,,
\end{equation} see \cite{daley}, p. 232.
Using the method of Poisson embedding, originally due to \cite{kerstan}, \cite{BreMas} (Section 3) obtained   general results on stability of Hawkes processes, even in the  non--linear case.

\section{Central limit theorem}\label{sec:CLT}

As explained in Section 2, the total claim amount for claims, arriving before time
$t$,  can be written as
\[
S(t) = \dsum_{\tau_k \leq t} f(A_k) = \dint_{0}^t \dint_\mathbb{S} f(u) N (ds,du)\,.
\]

The long term behavior of $S(t)$  for general marked Poisson cluster processes
 is the main goal of our study.
As before, by $Q$  we denote the probability distribution of marks on the space $\mathbb{S}$.

Moreover, unless stated otherwise, we assume that the process starts from 0 at time $t=0$, that is $N (-\infty, 0] = 0$. 

In the case of the Hawkes process,
the process $N_t = N([0,t]
\times \mathbb{S}) ,\ t \geq 0$  which only counts the arrival of claims until time $t$ has been  studied in the literature before. It was shown recently under appropriate moment conditions, that in the unmarked case  multitype Hawkes processes satisfy (functional) central limit theorem,
see ~\cite{bac}. \cite{zhu} showed that $N_t$ satisfies central limit theorem even in the more general case of nonlinear Hawkes process and that linear but marked Hawkes have the same property. In the present
section we describe the asymptotic behaviour of the total claim amount process $(S(t))$ for a wide class of marked Poisson cluster processes, even in the case when the total claim process has  heavy tails, and potentially infinite variance or infinite mean.

It is useful in the sequel to introduce random variable
$$
 \tau(t)  = \inf \left\{ n : \taun > t \right\},\ t \geq 0\,.
$$
Recall from \eqref{e:Di} the definition of $D_i$ as
$$
 D_i = \dint_{[0,\infty)\times \mathbb{S}}   f(u) C_i (ds,du)
 = \dsum_{j=0}^{K_i} f(A_{ij}) = \dsum_{j=0}^{K_i} X_{ij}   \,,
 $$
 where $K_i+1 = C_i[0,\infty)$ denotes the size of the $i$th cluster and where we denote 
 $X_{ij} = f(A_{ij})$.
As before, $D_i$ has an interpretation as the total claim amount coming from the $i$th immigrant and
its progeny. 
Note that $D_i$'s form an i.i.d. sequence because the ancestral mark in every cluster comes from an independently marked homogeneous Poisson point process.

Observe that in the nonstationary case we can write
\begin{equation} \label{e:Soft}
 S(t) = \sum_{i=1}^{\tau(t)} D_i 
  - D_{\tau(t)}  - \vep_t\,,\ t \geq 0\,,
\end{equation}
where the last error term represents the leftover or the residue at time $t$, i.e.  the sum of all  the claims arriving after $t$ which belong to the progeny of immigrants arriving before time $t$, that is
$$
 \vep _t =  \dsum_{0\leq \taui \leq t ,  t < \taui+T_{ij}}  f(A_{ij})  \quad t \geq 0\,.
$$

Clearly, in order to characterize limiting behaviour of $S(t)$, it is useful to determine moments and the tail behaviour of random variables $D_i$ for each individual cluster model. 
To simplify the notation, for a generic member of an identically distributed sequence or an array, say $(D_n)$, $(A_{ij})$,  we write $D,\, A\,$ etc.
Under the conditions of existence of second order moments and the behavior of the residue term $\vep_t$, it is not difficult to derive the following proposition. 

\begin{prop} \label{prop:CLT}
Assume the marked Poisson cluster model defined in Section 2. Suppose that $\Exp D^2 <\infty$ and that 
$\vep_t = o_P(\sqrt{t}) $ then,
for $t\toi$, 
\begin{equation}\label{eq:CLT}
 \dfrac{S(t) - t \nu \mu_D}{\sqrt{t\nu \Exp D^2} } \cid N(0,1)\,,
\end{equation}
where $\mu_D = \Exp D$.
\end{prop}

\begin{proof}
Denote the first term on the r.h.s. of \eqref{e:Soft}   by
$$
S^D(t) = \sum_{i=1}^{\tau(t)} D_i \, \ t \geq 0\,.
$$
An application of the central limit theorem for two-dimensional random walks, see \cite[Section 4.2, Theorem 2.3]{gut}   yields
$$
 \dfrac{S^D(t) - t \nu \mu_D}{\sqrt{t\nu \Exp D^2} } \cid N(0,1)\,,
$$
as $t\toi$. Since we assumed ${\vep_{t}}/{\sqrt t} \cip 0$, it remains to show that
$$
 \dfrac{D_{\tau(t)}}{\sqrt t} \cip 0 \quad \,t\to \infty.
$$

However, this follows at once from \cite[Theorem 1.2.3]{gut} for instance, or from the fact that in this setting sequences $(\taun)$ and $(D_n)$ are independent.

\end{proof}

 Note that \eqref{prop:CLT} holds for $f$ taking possibly negative values as well. However, when modelling insurance claims, non-negativity assumption seems completely natural, and in  the heavy tail case our proofs actually depend on it, cf. the proof of Theorem~\ref{FCLT:01}.
 In the special case $f\equiv 1$, one obtains the central limit theorem for the number of arrivals in time interval 
 $[0,t]$. Related results have appeared in the literature before,  see for instance \cite{daley72}
 or \cite{zhu}\,. 
 The short proof above stems from the classical Anscombe's theorem, as presented in \cite[Chapter IV]{gut} (cf. \cite[Theorem 3 ii]{daley72})  unlike the argument in \cite{zhu} which relies on martingale central limit theorem and seems not easily extendible, especially for heavy tailed claims we consider next. 
\begin{remark}\label{rem:counterex}It is not too difficult to find examples where the residue term is not negligible. Consider renewal cluster model of subsection \ref{ss:modren} with $K=1, \, X=1$. Let $W_{i1}$ be i.i.d. and regularly varying with index $\alpha < {1}/{2}.$ Then $\varepsilon_t$ has Poisson distribution with parameter $\EE [W \1_ {W<t}] \toi $ and thus, by Karamata's theorem, ${\varepsilon_t}/{\sqrt{t}}$ tends to infinity in probability. Similarly, one can show that $\Var({\varepsilon_t}/{\sqrt{t}})=\EE [W \1_ {W<t}]/t\to 0$ so that $(\varepsilon_t-\EE [W \1_ {W<t}])/\sqrt t$  tends to zero in probability. Thus
		\eqref{eq:CLT} does not hold any more but instead we have
		$$
 \dfrac{S(t) - t \nu \mu_D+\EE [W \1_ {W<t}]}{\sqrt{t\nu \Exp D^2} } \cid N(0,1)\,,\qquad t\to \infty\,.
 $$		
	\end{remark}

\section{Infinite variance stable limit}\label{sec:IVSL}

	It is known that if the claims are sufficiently heavy tailed, properly scaled and centred sums $S(t)$ may converge to
	an infinite variance stable random variable. In the case of
	random sums $S_n= X_1+\cdots +X_n$ of i.i.d. random variables, 
	the corresponding statement  is true if and only if the claims are regularly varying with index $\alpha\in (0,2)$.
	For the  Cram\'er--Lundberg model, i.e. when $N=N_0$, with i.i.d. regularly varying claims of index $\alpha \in (1,2)$,  corresponding limit theorem follows from Theorem 4.4.3 in \cite{gut}.
	A crucial step
	in the investigation of the heavy tailed case is to determine the tail behaviour of the random variables of \eqref{e:Di}.

For regularly varying $D_i$ with index $\alpha \in (1,2)$,  limit theory for two-dimensional random walks in Section 4.2 of \cite{gut} still applies.
Note, if one can show that $D_i$'s have regularly varying distribution, then there exists a sequence $(a_n)\,,\  a_n \toi$, such that
$$
n P(D>a_n)  \to 1\,, \qquad n\to\infty,
$$ 
and an $\alpha$--stable random variable $G_\alpha$ such that
$
{S^D_n} = D_1+\cdots +D_n\,,\qquad n\to\infty,
$
satisfies 
\begin{equation}\label{e:SDG}
\frac{S^D_n - n \mu_D }{a_n} \cid G_\alpha\,,
\end{equation}
where $\mu_D = \Exp D_i$. It is also known that the sequence $(a_n)$ is regularly varying itself with index $1/\alpha$, see \cite{res87}.
In the sequel ,we also set $a_t = a_{\lfloor t \rfloor }$
for any $t \geq 1$\,.

\subsection{Case $\alpha \in (1,2)$}

In this case, the arguments  of the  previous section can be adopted to show.

\begin{prop} \label{prop:12}
	Assume the marked Poisson cluster model introduced in Section 2. Suppose that $ D_i$'s are regularly varying with index $\alpha \in (1,2)$ and that 
	$\vep_t = o_P(a_{t}) $, then 
there exists an $\alpha$--stable random variable $G_\alpha$ such that 
	for $\mu_D = \Exp D_i$
\begin{equation}\label{eq:RV1}
	\frac{S(t) - t \nu \mu_D }{ 
		a_{\nu t } }
	\cid G_\alpha\,,
\end{equation}
as $t\toi$\,.
\end{prop}

\begin{proof}
	The proof again follows from the representation
	 \eqref{e:Soft},
	 by an application of Theorem~4.2.6 from \cite{gut}
	on random walks $(\Gamma_n)$ and $(S^D_n)$
	together with  relation \eqref{e:SDG}.
	By assumption we have $\vep_t/a_{\nu t}\sim \nu^{-1/\alpha} \vep_t /a_t {\cip} 0$.
To finish the proof, we observe that the sequences $(\taun)$ and $(D_n)$ are independent, hence
$$
\dfrac{D_{\tau(t)}}{a_{\nu t}} \cip 0, \qquad t\to \infty.
$$
\end{proof}

\subsection{Case $\alpha \in (0,1)$}

In this case, we were not able to find any result of Anscombe's theorem type for two--dimensional random walks of the type used above. Therefore, as our initial step, we prove a theorem which we believe is new and of independent interest. It concerns partial sums of i.i.d. nonnegative regularly varying random variables,
say $(Y_n)$, subordinated to an independent renewal process.
More precisely, set
$V_n= Y_1+\cdots + Y_n\,, n\geq 1$. Suppose that the sequence $(Y_n)$ is independent  of another i.i.d. sequence of nonnegative and nontrivial random variables $(W_n)$. Denote by 
\[
\sigma(t) = \sup \{ k: W_1+\cdots + W_k \leq t \}
\]
the corresponding renewal process,
where we set $\sup \emptyset = 0$\,. 
Recall that  for regularly varying random variables $Y_i$'s there exists a sequence $(a_n)$ such that
$n  \PP (Y_i>a_n) \to 1$, as $n \toi$\,.
The limiting behaviour of the process $V_{\sigma(t)}$ was considered by \cite{anderson} in the case when $W_i's$ are themselves regularly varying with index $\leq 1.$

Since ${\sigma(t)}/{t} \cas \nu,$ if $0<\EE W_i = {1}/{\nu} < \infty,$ one may expect that $V_{\sigma(t)}$ has similar asymptotic behaviour as $V_{\nu t}$ 	for $t \rightarrow \infty$. It is not too difficult to make this argument  rigorous 	
 if for instance $\EE W_i^2  < \infty,$ because then
	$ ( \sigma(t) - t \nu )^2/t,\, t>0$, is uniformly integrable 
  by Gut (2009), Section 2.5. 
	The following functional limit theorem gives precise description of the asymptotic behaviour of $V_{\sigma(t)}$ whenever $W_i$ have a finite mean.
\begin{thm} \label{FCLT:01}
Suppose that  $(Y_n)$ and  $(W_n)$ are independent {nonnegative} i.i.d. sequences of random variables such that $Y_i$'s are regularly varying  with  index $\alpha \in (0,1)$,
and such that $ 0< 1/\nu = \Exp W_i < \infty $.
Then in the space $D[0,\infty)$ endowed with Skorohod's $J_1$ topology
\begin{equation}\label{e:FLT}
	\frac{V_{\sigma(t\cdot)} }{a_{ \nu t } }
	\cid G_\alpha (\cdot)\,,\qquad t\to \infty,
\end{equation}
	where $(G_\alpha(s))_{s\geq 0}$ is an $\alpha$--stable subordinator.
\end{thm} 

\begin{proof}
	 Since $Y_i$'s are regularly varying, it is known,
	 \cite{res87,res07}, that
	  the following point process convergence  holds as $t \toi$
\begin{equation}\label{e:PRMcon}
	  M'_n= \sum_i \delta_{\frac in, \frac{Y_i}{a_n}} \cid  M_{\alpha}
	  \sim \mbox{PRM} (Leb \times d (- y^{-\alpha})) \,, 
\end{equation}
	  with respect to the vague topology on the space
	  of Radon point measures on $[0,\infty)\times (0,\infty]$.
	 Abbreviation PRM  stands for  Poisson random measure indicating that the limit is a Poisson process. 
	 Starting from \eqref{e:PRMcon}, it was shown in \cite[Chapter 7]{res07} for instance, that
	 for an $\alpha$--stable subordinator $G_\alpha(\cdot)$ as in the statement of the theorem
	   \begin{equation}\label{eq:j1}
	   V'_n (\cdot) = \frac{V_{\lfloor n \cdot \rfloor}}{a_n} \cid G_\alpha (\cdot)\,,\qquad t\to \infty,
	   \end{equation}
	   in Skorohod's $J_1$ topology on the space $ D[0,\infty)$.  Observe that
	   since $\alpha \in (0,1)$, no centering is needed, and that 
	   one can  substitute the integer index $n$ by a continuous index $t \toi$. Note further that  we  have the joint convergence
	 \begin{equation} \label{e:MVcon}
	    (M'_t,V'_t) \cid (M_\alpha, G_\alpha),\qquad t\to \infty,
	 \end{equation}
	   in the product topology on the space of point measures and c\`adl\`ag functions.
	  Moreover, it is known that
	  the jump times and sizes of the $\alpha$--stable subordinator  $G_\alpha$  correspond to the points of the limiting point process $M_{\alpha}$\,.

The space of point measures and the space of c\`adl\`ag functions $D[0,\infty)$ are both Polish, in respective  topologies,  therefore,
Skorohod's representation theorem applies. Thus, we can assume
	that convergence in \eqref{e:MVcon} holds a.s.
	 on a certain probability space $(\Omega,\mathcal{F},P)$, and in particular 
	 there exists
	  $\Omega' \subseteq \Omega$, such that $P(\Omega') = 1$ and
for all $\omega \in \Omega'$,
	$V'_t  \to G_\alpha $ in $J_1$  and
	$M'_t \to M_\alpha$ in vague topology.
	By Chapter VI, Theorem 2.15 in \cite{JacShi}, for any such $\omega$
	there exists a dense set $B=B(\omega)$ of points in $[0,\infty)$ such that
	\[
	{V'_t}  (s) \to G_\alpha (s)\,,\qquad t\to \infty,
	\]
	for every $s \in B $\,, where actually $B $ is simply the set of all nonjump  times in the path of the process $G_\alpha $.  
	   On the other hand, it is known that in $J_1$ topology,
	   on some set $\Omega''$ such that $P(\Omega'') = 1$,
	  \begin{equation} \label{e:taucas}
	    \frac{\sigma(t\cdot)}{t\nu}  \to id(\cdot)\,,\qquad t\to \infty,
	  \end{equation}
where $id$ stands for the identity map. This follows directly by an application of Theorem 2.15 in Chapter VI of \cite{JacShi}.
 Moreover, by Proposition VI.1.17 in ~\cite{JacShi}, the convergence in \eqref{e:taucas}  holds locally uniformly on $D[0,\infty)$.
	 
Consider now for fixed $t>0$ and $\omega \in \Omega' \cap \Omega''$
\[
 V_t  (s) = \dfrac{V_{\sigma(t  s)  }}{a_{\nu t}}\,,
 \quad s \geq 0\,.
\]	 
From \eqref{e:taucas} we may expect that
$V_t  (s) \approx V'_{t\nu } (s).$
Indeed, for any fixed $0<\delta<1$ and all large $t$, we know that
$ \lfloor tc\nu (1-\delta) \rfloor\leq\sigma(tc)
\leq \lfloor tc\nu (1+\delta) \rfloor$,
which by monotonicity of the sums implies
\[
  \dfrac{V_{\lfloor tc\nu (1-\delta) \rfloor} }{a_{\nu t}}
  \leq\dfrac{V_{\sigma(t  c)} }{a_{\nu t}}
  \leq\dfrac{V_{\lfloor tc\nu (1+\delta) \rfloor } }{a_{\nu t}}\,.
\]
Now, for $c(1-\delta) $ and $c(1+\delta)$ in $B$, the left hand side and the right hand side above converge
to
$G_\alpha(c(1-\delta))$ and $G_\alpha(c(1+\delta))$.
Thus, if we consider $c\in B$ and let $\delta \to 0$,  then
\begin{equation} \label{e:V2Gcon}
\dfrac{V_{\sigma(t c)} }{a_{\nu t}} \to G_\alpha(c)\,,\qquad t\to \infty,
\end{equation}
for all $\omega \in \Omega' \cap \Omega''$ and thus with probability $1$.

By Theorem 2.15 in Chapter VI in \cite{JacShi}, to prove \eqref{e:FLT}, it remains to show that for all $\omega \in \Omega' \cap \Omega''$ and $c\in B$, as $t \toi$
\begin{equation} \label{e:DeltaV}
	 \dsum_{0<s\leq c} | \Delta V_t(s)|^2
	 = \dsum_{i<\sigma(tc) } \left(\frac{Y_i}{a_{t\nu}}\right)^2
	 \to  \dsum_{0<s\leq c} | \Delta G_\alpha(s)|^2 \,,
\end{equation}	
where, for an arbitrary c\`adl\`ag process $X(t)$ at time $t \geq 0$, we denote $\Delta X(t) = X_{t} - X_{t-}\,.$
Observe that 
	$$
	 G_{\alpha/2}(c):=\dsum_{0<s\leq c} | \Delta G_\alpha(s)|^2
	$$
	defines an $\alpha/2$--stable subordinator and that the squared random variables
 $Y_i^2$ are again regularly varying with index $\alpha/2$ with the property that
 	$n \PP(Y_i^2 >a^2_n) \to 1$. A similar approximation argument as for \eqref{e:V2Gcon} shows that \eqref{e:DeltaV} indeed holds, which concludes the proof.
\end{proof}

Assume now that  $P(D>x)= x^{-\alpha} \ell(x)$ for some slowly varying function $\ell$ and $\alpha \in(0,1)$.
Select a sequence $a_n \toi$ such that
$
n P(D>a_n)  \to 1\,, 
$ 
as $n\toi$. Under suitable conditions on the residue
term $\vep_t$ we obtain the following.

\begin{prop} \label{prop:01}
	Assume that $ D_i$'s are regularly varying with index $\alpha \in (0,1)$ and that 
	$\vep_t = o_P(a_{t}) $. Then, 
there exists an $\alpha$--stable random variable $G_\alpha$ such that 
\begin{equation}\label{eq:RV2}
	\frac{S(t) }{ 
		a_{\nu t } }
	\cid G_\alpha\,,
\end{equation}
	as $t\toi$\,.
\end{prop}
\begin{proof}
	The proof follows roughly the same lines as the proof of Proposition~\ref{prop:12}, but here we rely on an application of the previous theorem to the  random walks $(\Gamma_n)$ and $(S^D_n)$. Just, instead of $Y_i$'s and $W_i$'s we have  $D_i$'s and an independent sequence of i.i.d. exponential random variables with parameter $\nu$\,.

\end{proof}

\begin{remark}\label{statvers} 
One can consider total claim amount in the period $[0,t]$  for  the stationary model of subsection~\ref{sss:STAT}, i.e.
\[
S^*(t) = \dint_{0}^t \dint_\mathbb{S} f(u) N^* (ds,du)\,,\qquad t \geq 0\,.
\]
 Here again,   $S^*(t)$  has a similar representation as in \eqref{e:Soft} but with an additional term on the right hand side, i.e.
    \begin{equation} \label{e:Soft3}
  S^*(t) = \sum_{i=1}^{\tau(t)} D_i 
  - D_{\tau(t)}  - {\vep}_t + {\vep}^*_{0,t}\,,\qquad t \geq 0\,,
  \end{equation}
  where 
    $$
  \vep^*_{0,t} = \dsum_{\Gamma_i \leq 0 ,\  0 < \Gamma_i+T_{ij}<t}  X_{ij}.
  $$
  Clearly, by stationarity 
\begin{equation} \label{eq:vepminus}
  \vep_t = \dsum_{0\leq \taui \leq t ,  t < \taui+T_{ij}}  X_{ij} \eind \vep^{-}_t = \dsum_{-t\le \taui \leq 0 ,  0 < \taui+T_{ij}}  X_{ij}.
\end{equation}
  Hence, $\vep_t = o_P(a_{t}) $ yields  $\vep^-_t = o_P(a_{t}) $ for any sequence $(a_t)$ and therefore
  \[
  \vep^*_{0,t} \leq \vep_t ^-+ \dsum_{\Gamma_i < -t ,\  0 < \Gamma_i+T_{ij}<t}  X_{ij}= \dsum_{\Gamma_i < -t ,\  0 < \Gamma_i+T_{ij}<t}  X_{ij}+o_P(a_{t})\,.
  \]
  In particular, conclusions of propositions~\ref{prop:CLT}, \ref{prop:12} and \ref{prop:01}  hold for random variables $S^*(t)$ too under the additional assumption that
  \begin{equation}\label{eq:adas}
  \tilde \vep_t:=\dsum_{\Gamma_i \leq -t ,\  0 < \Gamma_i+T_{ij}<t}  X_{ij} =o_P(a_{t}).
\end{equation}
\end{remark}

\section{Total claim amount for special models}\label{sec:APP} 

 As we have seen in the previous two sections, it is relatively easy to describe asymptotic behaviour of the total claim amount $S(t)$ as long as we are able to determine the moments and tail properties of the random variables $D_i$ and the residue random variable $\vep_t$ in \eqref{e:Soft} (and also $\tilde \vep_t$ in \eqref{eq:adas} for the stationary version). However, this is typically a rather technical task, highly dependent on an individual Poisson cluster model. 
 In this section we revisit three models introduced in
 Subsection~\ref{ss:Mod}, characterizing for each of them the limiting distribution of the total claim amount under appropriate conditions. Note  that the cluster sum $D$ for all three models admits the following representation
\[
D\stackrel{d}{=}\sum_{j=0}^K X_j,
\]
for $(X_j)_{j\ge 0}$ i.i.d. copies of $f(A)$ and some integer valued $K$ such that $\EE[K_1]<\infty$. Throughout, we assume that the random variables $K$ and $(X_j)_{j\ge 1}$ are independent. The sum  $\sum_{j=1}^KX_j$ has a  so called compound distribution. Its first two moments exist under the following conditions 
\begin{itemize}
\item if $\EE[X]<\infty$ and $\EE[K]<\infty$, then $\mu_D = \EE D=(1+\EE[K])\EE[X] < \infty$,
\item if $\EE[X^2]<\infty$ and $\EE[K^2]<\infty$, then $\EE D^2=(\EE[K]+1)\EE[X^2]+(\EE[K^2]+\EE[K])\EE[X]^2 < \infty$. 
\end{itemize}
The tail  behaviour of compound sums was often studied under various conditions (see \cite{RoSe,fayetal,HuSa,denisov}). We list below some of these conditions, which are applicable to our setting.
\begin{itemize}
\item[{\bf (RV1)}] If $X$ is regularly varying with index $\alpha >0$  and $\PP(K>x)=o(\PP(X>x))$, then $\PP(D>x) \sim (\EE[K]+1)\PP(X>x)$ as $x\to \infty$, see \cite[Proposition 4.1]{fayetal},
\item[{\bf (RV2)}] If $K$ is regularly varying with index $\alpha \in (1,2)$ and $\PP(X>x)=o(\PP(K>x))$, then $\PP(D>x) \sim \PP(K>x/\EE[X])$ as $x\to \infty$, see \cite[Theorem 3.2]{RoSe} or \cite[Proposition 4.3]{fayetal},
\item[{\bf (RV3)}] If $X$ and $K$   are both regularly varying with index $\alpha \in (1,2)$ and tail equivalent, see \cite[Definition 3.3.3]{embrechtsetal}, then $\PP(D>x) \sim (\EE[K]+1)\PP(X>x) + \PP(K>x/\EE[X])$ as $x\to \infty$, \cite[Theorem 7]{denisov}.
\end{itemize}
We will refer to the last three conditions as the sufficient conditions {\bf (RV)}.
  
\subsection{Mixed binomial cluster model}\label{sec:mixbin}
Recall from subsection \ref{ss:modbin} that the clusters in this model have the following form
$$
G^{A_i} = \dsum_{j=1}^{K_i} \delta_{W_{ij}, A_{ij}}\,.
$$
Assume:
\begin{itemize} 
	\item $(K_i,(W_{ij})_{j\ge 1},(A_{ij})_{j\ge 0})_{i\ge 0}$ constitutes an i.i.d. sequence\,,
 \item $(A_{ij})_{j\ge0}$ are i.i.d. for any fixed $i$\,,
 \item $(A_{ij})_{j\ge1}$ is independent of $K_i,(W_{ij})_{j\ge 1}$ for all $i\ge 0$\,,
 \item $(W_{ij})_{j\ge 1}$ are conditionally i.i.d. and independent of $K_i$ given  $A_{i0}$.
\end{itemize}
Thus we do not exclude the possibility of dependence between $K_i,(W_{ij})_{j\ge 1}$ and the ancestral mark $A_{i0}$. For any $\gamma>0$, we  denote  by 
\[
A,\,X_j,\, K, W_j,\, m_A\,,m_A^{(\gamma)}\,,
\] generic random variables
 with the same distribution as
 $A_{ij},\,X_{ij}=f(A_{ij}),\, K_i, W_{ij}$, $\EE[K_i\mid A_{i0}]$ and $\EE[K_i^\gamma\mid A_{i0}]$ respectively.  Using the cluster representation, one can derive the asymptotic properties of $S(t)$. Let us first consider the 
Gaussian CLT under appropriate 2nd moment assumptions.  Denote by $\PP(W \in \cdot \mid A)$ the distribution of $W_{ij}$'s given $A_{i0}$. 
\begin{corollary}\label{cor:CLTmix}
Assume that $\EE[X^2]<\infty$ and $\EE[K^2]<\infty$. If
\begin{equation}\label{eq:condCLTmix}
\sqrt t \EE[m_A\PP(W>t\mid A)] \to 0,\qquad t\to \infty,
\end{equation}
then the relation 
\eqref{eq:CLT} holds.
\end{corollary}
Observe  that \eqref{eq:condCLTmix} is slightly weaker than the existence of the moment  $\EE[K \sqrt W]<\infty$.
\begin{proof}
It follows from the compound sum representation of $D$ that $\EE D^2<\infty$ as soon as $\EE[X^2]<\infty$ and $\EE[K^2]<\infty$. By Proposition \ref{prop:CLT}, it remains  to show that $\vep_t = o_P(\sqrt t)$. In order to do so, we use the Markov inequality 
\[
\PP(\vep_t>\sqrt t)\le \frac{\EE[\vep_t]}{\sqrt t}=\frac{\EE\left[\sum_{0\le \Gamma_i\le t}\sum_{j=1}^{K_i}\1_{t< \Gamma_i+W_{ij}}f(A_{ij})\right]}{\sqrt t}.
\]
We use Lemma 7.2.12 of \cite{mikosch} with $f(s) = \sum_{j=1}^{K_i}\1_{W_{ij}> t-s} f(A_{ij})$ in order to compute the r.h.s. term as
\begin{align*}
\frac{\int_0^t\EE\left[\sum_{j=1}^{K_i}\1_{W_{ij}> t-s} f(A_{ij})\right]\nu ds}{\sqrt t}
&=\frac{\nu\EE[X]\int_0^t\EE\big[\EE\big[ \sum_{j=1}^{K_i}\1_{W_{ij}> t-s} \mid A_{i0}\big]\big]ds}{\sqrt t}\\
&=\frac{\nu\EE[X]\int_0^t\EE[m_A\PP(W> x\mid A)]dx}{\sqrt t}.
\end{align*}
Notice that the last identity is obtained thanks to the  independence of $K_i$ and $(W_{ij})_{j\ge 0}$  conditionally on $A_{i0}$.
We conclude by the L'H\^opital's rule  that this converges to $0$ under \eqref{eq:condCLTmix}.
\end{proof}
For regularly varying  $D$ of order $1<\alpha<2$, we obtain the corresponding limit theorem under weaker assumptions on the tail of the waiting time $W$.
\begin{corollary} \label{cor:stable1mix}
Assume that one of the  conditions {\bf (RV)} holds for $1<\alpha<2$, so that  $D$ is regularly varying. When
\begin{equation}\label{eq:condstable1mix}
t^{1+\delta-1/\alpha} \EE[m_A\PP(W>t\mid A)] \to 0,\qquad t\to \infty,
\end{equation}
for some $\delta>0$ the relation 
\eqref{eq:RV1} holds.
\end{corollary}
The condition \eqref{eq:condstable1mix} is  slightly weaker than assuming $\EE[m_A W^{1+\delta-1/\alpha}]<\infty$.  Notice that when $\alpha\to 1^+$ and $K$ is independent of $W$, this condition boils down  to the existence of a $\delta$'th moment of $W$ for any strictly positive $\delta$. 
\begin{proof}
By definition,  $(a_t)$ satisfies $t\PP(D>a_t)\to 1$ as $t\to \infty$ and $(a_t)$ is regularly varying with index  $1/\alpha$. Applying the Markov inequality as in the proof of Corollary \ref{cor:CLTmix}, we obtain
\[
\PP(\vep_t>a_t)\le \frac{\EE[\vep_t]}{a_t}=\frac{\nu\EE[X]\int_0^t\EE[m_A\PP(W> s\mid A)]ds}{a_t}.
\]
The claim follows now by the L'H\^opital's rule and the relation $t^{1/\alpha-\delta} = o(a_t)$ for any $\delta>0$.
\end{proof}

\begin{remark}\label{statvers:binom}
In the context of the mixed binomial model, consider the total claim amount of the stationary process denoted by $S^*(t)$ which takes into account also the arrivals in the interval $(-\infty,0)$, see Remark~\ref{statvers}. Assume for simplicity that $K_i$'s and $(W_{ij})$'s are unconditionally  independent. Then $\tilde \vep_t$ from \eqref{eq:adas} is $o_P(a_{t})$ under the same conditions as in Corollaries~\ref{cor:CLTmix} and \ref{cor:stable1mix}, where we set $a_t = \sqrt t$ in the former case. Indeed, we will show that
$$
\EE \tilde \vep_t
= \EE \left(\dsum_{\Gamma_i \leq -t ,\  0 < \Gamma_i+W_{ij}<t}  X_{ij}\right) = 
\EE  \left(\dsum_{-t\le \taui \leq 0 ,  t < \taui+W_{ij}}  X_{ij}
 \right)\,,
$$
 so that $\EE \tilde \vep_t= o_P(a_{t})$ as well since the r.h.s.  is dominated by $\EE \vep^-_t  = o(a_t)\,,$
cf.~\eqref{eq:vepminus}.

 Note first that under assumption of the last two corollaries, individual claims have finite expectation, i.e. $\EE X <\infty$. So it suffices to show that
$$
 I_1 := \EE \tilde \vep_t/ \EE X =  
 \EE  \dsum_{\Gamma_i < -t } \dsum_{j=1}^{K_i} \1_{  0 < \Gamma_i+W_{ij}<t} 
 =  \EE  \dsum_{-t<\Gamma_i < 0 } \dsum_{j=1}^{K_i} \1_{\Gamma_i+W_{ij}>t} 
 =: I_2   \,. 
$$
From $I_1,I_2$ we  subtract respectively l.h.s. and r.h.s. of the equality
\[
 \EE  \dsum_{-2t<\Gamma_i < -t } \dsum_{j=1}^{K_i}  \1_{\Gamma_i+W_{ij} \in (0,t) } \1_{W_{ij}\in (t,2t]}  =
  \EE  \dsum_{-t<\Gamma_i < 0 } \dsum_{j=1}^{K_i}  \1_{\Gamma_i+W_{ij} \in (t,2t) } \1_{W_{ij}\in (t,2t]}\,,
\] 
where the equality follows by the stationarity of the underlying Poisson process,
to obtain 
$$
  J_1 = \EE  \dsum_{-\infty<\Gamma_i < -t } \dsum_{j=1}^{K_i}  \1_{0<\Gamma_i+W_{ij} < t } \1_{W_{ij} > 2t }
   = \EE K \dint_{-\infty}^{-t} \nu ds \dint_{-s \vee 2t }^{t-s} dF_W(u) 
$$
and 
$$
    J_2 = \EE  \dsum_{-t <\Gamma_i < 0 } \dsum_{j=1}^{K_i}  \1_{\Gamma_i+W_{ij} > t } \1_{W_{ij} > 2t }
    = \EE K  \dint_{-t}^{0} \nu ds \dint_{2t}^\infty dF_W(u) 
$$
where $F_W$ denotes the distribution function of delays $(W_{ij})$.
Finally, note that 
$$
 J_1 = \EE K   \dint_{ 2t }^{\infty} dF_W(u)   \dint_{-u}^{t-u} \nu ds = J_2\,.
$$
\end{remark}

Since we assumed that $\EE[K]<\infty$, the regular variation property of $D$ with index $\alpha\in(0,1)$ can arise only through the claim size distribution, see Proposition 4.8 in \cite{fayetal}. It turns out that in such a heavy tailed case,  no additional assumption on the waiting time $W$ is needed.

\begin{corollary} \label{cor:stable2mix}
Assume that $X$ is regularly varying of order $0<\alpha<1$, then the relation  \eqref{eq:RV2} holds.
\end{corollary}
\begin{proof}
Observe that one cannot apply Markov inequality anymore because $\EE D=\infty$. Instead, we use the fact that $\sum_{j=1}^tX_j/a_t$ converges because $X$ and $D$ have equivalent regular varying tails. Recall from \eqref{eq:vepminus} that 
  $$\vep_t 
  = \dsum_{0\le \taui \leq t ,  t < \taui+W_{ij}}  X_{ij}.$$
We denote the (increasing) number of summands in the r.h.s. term by $M_t=\# \{i,j:0\le \taui \leq t , t < \taui + W_{ij}\}$. We can apply Proposition \ref{prop:01} after observing that $\sum_{j=1}^{M_t}X_j/a_{M_t}$ is a tight family of random variables, because $M_t$ is independent of the array $(X_{ij})$. 
Writing
\begin{equation}\label{eq:ratios}
\frac{\vep_t}{a_t}\stackrel{d}{=}\frac{\sum_{j=1}^{M_t}X_j}{a_{M_t}}\frac{a_{M_t}}{a_t}\,,
\end{equation}
and observing that $a_t$ is regularly varying with index $1/\alpha$, we obtain the desired result provided  that $M_t=o_P(t)$. It is sufficient to show the convergence to $0$ of the ratio

\begin{align*}
\frac{\EE [M_t]}{t}&=\frac{\EE [\# \{i,j:0\le \taui \leq t , t< \taui + W_{ij}\}]}{t} \\
& =\frac{\EE\left[\sum_{0\le \taui \leq t }\sum_{j=1}^{K_i}\1_{t\le  \taui +W_{ij}}\right]}{t}.
\end{align*}

Using similar calculation as in the proof of Corollary \ref{cor:CLTmix} (setting $X=1$), we obtain an explicit formula for the r.h.s. term as
\[
\frac{\nu  \int_0^t\EE[m_A\PP(W>x\mid A)]dx}{t}\to 0, \qquad t\to \infty,
\]
the convergence to $0$ following from a Cesar\`o argument.
\end{proof}

\subsection{Renewal cluster model}
Recall from subsection \ref{ss:modren} that the clusters of this model have the following form
		$$
		G^{A_i} = \dsum_{j=1}^{K_i} \delta_{T_{ij}, A_{ij}}\,,
		$$
	where 
	\begin{itemize} 
		\item $T_{ij}=W_{i1}+\cdots +W_{ij}$\,,
		\item while $(K_i,(W_{ij})_{j\ge 1},(A_{ij})_{j\ge 0})_{i\ge 0}$ constitutes an i.i.d. sequence satisfying the assumptions listed in Section \ref{sec:mixbin}\,.
	\end{itemize}

The total claim amount coming from the $i$th immigrant and its progeny is again
\[
D \stackrel{d}{=}\sum_{j=0}^K X_j,
\]
for $(X_j)$ i.i.d. copies of $f(A)$. Dealing with the waiting times $T_{ij}=W_{i1}+\cdots +W_{ij}$ requires additional care than in the previous model. We obtain first
\begin{corollary} \label{cor:CLTRenewal} 
Suppose  $\Exp X^2  <\infty,$ $\Exp K^2  <\infty$ and 
$ \Exp [K^{2}W^\delta] < \infty$\, for some $\delta>1/2 $ then the relation \eqref{eq:CLT} holds.
\end{corollary}

\begin{proof}
	The proof follows from Proposition~\ref{prop:CLT}.
	Second moment of $D_i$'s is finite by the moment assumptions on $X$ and  $K$. It remains to show that the residue term satisfies 
	$\vep_t = o_P(\sqrt{t}) $. Using Lemma 7.2.12 of \cite{mikosch} with $f(x) = \sum_{j=1}^{K_i}\1_{W_{i1}+\cdots+W_{ij}> x} f(A_{ij})$ similarly as in the proof of Corollary \ref{cor:CLTmix} we obtain
\begin{align*}
\EE[\vep_t]&= \int_0^t\EE\left[\sum_{j=1}^{K_i}\1_{W_{i1}+\cdots+W_{ij}> x} f(A_{ij})\right]\nu dx \\
& = \nu \int_0^t\EE\left[\EE\left[\sum_{j=1}^{K_i}\1_{W_{i1}+\cdots+W_{ij}> x} f(A_{ij})\mid K_i, (W_{ij})_{j\ge 1}\right]\right] dx \\
& = \nu \EE[X]\int_0^t\EE\left[ \sum_{j=1}^{K_i}\1_{W_{i1}+\cdots+W_{ij}> x}  \right] dx
\end{align*}
by independence between $K_i, (W_{ij})_{j\ge 1}$ and $(f(A_{ij}))_{j\ge 1}$. The key argument in dealing with the renewal cluster model is the  following upper bound
\begin{equation}\label{eq:stochdom}
\sum_{j=1}^{K_i}\1_{W_{i1}+\cdots+W_{ij}>x}\le  \1_{ W_{i1}+\cdots+W_{iK_i}>x} K_i\,.
\end{equation}
Assume with no loss of generality that $\delta \le 1$.
By the Markov inequality and the  conditional  independence of $K_i$ and $(W_{ij})_{j\ge 0}$  conditionally on $A_{i0}$, we obtain

\begin{align}\nonumber
\EE\left[ \1_{W_{i1}+\cdots+W_{iK_i}> x}K_i\mid A_{i0}  \right]&\le \frac{\EE[K_i (W_{i1}+\cdots+W_{iK_i})^\delta\mid A_{i0}]}{x^\delta}\\
\label{eq:markov}
&\le m_{A_{i0}}^{(2)} \frac{\EE[W^\delta\mid A_{i0}]}{x^\delta}\,,
\end{align}
using the notation $m_{A_{i0}}^{(\gamma)}=\EE[K_i^\gamma \mid A_{i0}]$ for any $\gamma>0$.
The last inequality follows from the sub-linearity of the mapping $x\mapsto x^\delta$ for $\delta\le 1$. Thus, we obtain for some constant $C>0$
\[
\EE[\vep_t]\le \nu \EE[X] \EE[m_A^{(2)} W^\delta]\int_1^tx^{-\delta}dx +C=O(\EE[m_A^{(2)} W^\delta]t^{1-\delta})= o(\sqrt t)
\]
as $\delta>1/2$ by assumption.
\end{proof}

Regularly varying claims can be handled with additional care as $K$ may not be square integrable.
 
 \begin{corollary} \label{cor:RVRenewal} 
 Assume that one of the  conditions {\bf (RV)} holds so that  $D$ is regularly varying of order $1<\alpha<2.$ Suppose further that $\EE[K^{1+\gamma}]<\infty$ and $\EE[K^{1+\gamma}W^\delta]<\infty,$ $\delta>0,$ $\gamma >0$ and $ \delta>(\alpha-\gamma)/\alpha$. Then  the relation \eqref{eq:RV1} holds.
 \end{corollary} 
Observe that we obtain somewhat stronger conditions than in the mixed binomial case, see Corollary~\ref{cor:stable1mix} and remark following it.
 \begin{proof}
 	With no loss of generality we assume that $\gamma\le1$.
We use the Markov inequality of order $\gamma$
\[
\PP(\vep_t>a_t)\le \frac{\EE[\vep_t^\gamma]}{a_t^\gamma}.
\]
Thanks to the the sub-additivity of the function $x\mapsto x^\gamma$ we have
\begin{align*}
\EE[\vep_t^\gamma]&=  \EE\left[\left(\sum_{0\le \Gamma_i\le t}\sum_{j=1}^{K_i}\1_{\Gamma_i+W_{i1}+\cdots+W_{ij}>t}f(A_{ij})\right)^\gamma\right]\\
&\le  \EE\left[\sum_{0\le \Gamma_i\le t}\left(\sum_{j=1}^{K_i}\1_{\Gamma_i+ W_{i1}+\cdots+W_{ij}>t}f(A_{ij})\right)^\gamma\right]
\end{align*}
Using Lemma 7.2.12 of \cite{mikosch} with $f(x) = \sum_{j=1}^{K_i}\1_{W_{i1}+\cdots+W_{ij}>x}f(A_{ij})$ similarly as in the proof of Corollary \ref{cor:CLTmix} we obtain
\begin{equation}\label{eq:vepgamma}
\EE[\vep_t^\gamma]\le\nu \int_0^t\EE\left[ \left(\sum_{j=1}^{K_i}\1_{W_{i1}+\cdots+W_{ij}>x}f(A_{ij})\right)^\gamma\right] dx .
\end{equation}
We use Jensen's inequality as follows
\begin{multline*}
\EE\left[\EE\left[ \left(\sum_{j=1}^{K_i}\1_{W_{i1}+\cdots+W_{ij}>x}f(A_{ij})\right)^\gamma\mid K_i, (W_{ij})_{j\ge 1}\right]\right]\\
\le \EE\left[ \left(\EE\left[\sum_{j=1}^{K_i}\1_{W_{i1}+\cdots+W_{ij}>x}f(A_{ij})\mid K_i, (W_{ij})_{j\ge 1}\right]\right)^\gamma\right]
\end{multline*}
so that, using the independence between $K_i, (W_{ij})_{j\ge 1}$ and $(f(A_{ij}))_{j\ge 1}$, one gets
\[
\EE[\vep_t^\gamma]\le \nu \EE[X]^\gamma\int_0^t\EE\left[\left(\sum_{j=1}^{K_i}\1_{W_{i1}+\cdots+W_{ij}>x}\right)^\gamma\right]dx
\]
Using   the stochastic domination \eqref{eq:stochdom}, we obtain 
\[
\EE[\vep_t^\gamma]\le\nu\EE[X]^\gamma\int_0^t\EE\left[ \1_{ W_{i1}+\cdots+W_{iK_i}>x}K_i^\gamma\right]dx.
\]
With no loss of generality we assume $0<\delta<1$. Applying the Markov inequality of order $\delta$ conditionally on $A_{i0}$ as in \eqref{eq:markov}, we have
\[
\EE[\1_{ W_{i1}+\cdots+W_{iK_i}>x}K_i^\gamma\mid A_{i0}]\le  m_{A_{i0}}^{(1+\gamma)} \frac{\EE[W_i^\delta\mid A_{i0}]}{x^\delta}.
\]
Plugging in this bound in the previous inequality, we obtain for some $C>0$,
\[
\EE[\vep_t^\gamma]\le\nu\EE[X]^\gamma\EE\big[m_{A}^{(1+\gamma)}W^\delta\big]t^{1-\delta}+C =o(a_t^\gamma)
\]
as $1-\delta<\gamma/\alpha$ by assumption.
 \end{proof}

\begin{corollary}
If $X$ is regularly varying of order $\alpha \in (0,1)$ 
then  the relation 
\eqref{eq:RV2} holds.
\end{corollary}
\begin{proof}
We use the same arguments as in the proof of Corollary \ref{cor:stable2mix} in order to obtain \eqref{eq:ratios}. The desired result follows if one can show that $M_t=\# \{i,j:0\le \taui \leq t , t < \taui + W_{i1} + \cdots + W_{ij}\} = o_P(t)$. Using the Markov's inequality, it is enough to check that $\EE[M_t]/t=o(1)$. Following the same reasoning than in the proof of Corollary \ref{cor:RVRenewal}, we estimate the moment  of $M_t$ similarly as the one of $\vep_t$ in \eqref{eq:vepgamma}: 
\[
\EE[M_t]\le  \int_0^t\EE\left[  \sum_{j=1}^{K_i}\1_{W_{i1}+\cdots+W_{ij}>x}  \right]\nu dx\le  \int_0^t\EE\left[  K_i \1_{W_{i1}+\cdots+W_{iK_i}>x}  \right]\nu dx.
\]
We used again the stochastic domination \eqref{eq:stochdom} to obtain the last upper bound. From a Cesar\'o argument, the result will follow if 
\[
\EE\left[  K_i \1_{W_{i1}+\cdots+W_{iK_i}>x}  \right] \to 0,\qquad x\to \infty.
\]
One can actually check this negligibility property because the random sequence  $ K_i \1_{W_{i1}+\cdots+W_{iK_i}>x} \to 0$ a.s. by finiteness of $W_{i1}+\cdots+W_{iK_i}$ and because the sequence is dominated by $K_i$ that is integrable.
\end{proof}
 
\subsection{Marked Hawkes process}

Recall from Section \ref{ss:modhaw} that the clusters of the Hawkes model satisfy recursive relation \eqref{GA}. In other words, the clusters $G^{A_i}$ represent a recursive  aggregation of Poisson processes with random mean measure $ \tilde{\mu}_A  \times Q$\, which satisfies $\kappa =  \Exp \int h(s,A) ds < 1$.

In general, it is not entirely straightforward to see when the moments of $D$ are finite. However, note that $D_i$'s are i.i.d. and satisfy distributional equation 
\begin{equation}\label{eq:distD}
D \eind f(A) + \dsum_{j=1}^{L_A} D_j \,,
\end{equation}
where $L_A$ has the Poisson distribution conditionally on $A$, with
mean $\kappa_A=\int_0^\infty h(s,A) ds.$ Recall from (\ref{e:kappa}) that $\kappa = \Exp \kappa_A < 1.$ The $D_j$'s on the right hand side are independent of $\kappa_A$ and i.i.d. with  the same distribution as $D$. Conditionally on $A$, the waiting times are i.i.d. with common density
$h(t,A)/\kappa_A $, $t\ge0$. Thus, one can relate the clusters of the Hawkes process with those of a mixed binomial  process from Section \ref{sec:mixbin} with $K=L_A$. In order to obtain the asymptotic properties of $S(t)$ one still needs to characterize the moment and tail properties of $D$.\\

Consider  the Laplace transform of $D$, i.e.
$\varphi(s) = \Exp e^{-sD}$, for  $s \geq 0$. 
Also, recall the Laplace transform of a Poisson compound sum is of the form 
$$ \Exp \left[ 
		e^{-s \sum_{j=1}^{M}Z_j} \right] = \Exp \left[  e^{m_A \left( \Exp e^{-sZ} - 1\right)}  \right],
$$
where $M$ is Poiss$(m_A)$ distributed, independent of the i.i.d. sequence $(Z_i)$ of nonnegative random variables with common distribution (see, for instance, Section 7.2.2 in \cite{mikosch}).
Note, $\varphi$ is an infinitely  differentiable function  for  $s>0$. 
To simplify the notation, 
denote by
\[
X= f(A)\,,
\] a generic claim size and observe that  by \eqref{eq:distD}, $\varphi$ satisfies the following
\begin{align}
	 \varphi(s) & =  \Exp \left[ 
		\Exp\left( e^{-s(X + \sum_{j=1}^{L_A}D_j)}
		\bigg|\,  A
		\right)\right]  =
		\Exp \left[ 
		e^{-sX} \Exp\left(e^{ -s \sum_{j=1}^{L_A}D_j}
		\bigg|\,  A
		\right)\right] \nonumber \\
	& =  \Exp \left[ 
	e^{-sX} e^{\kappa_A  (\Exp e^{-sD} -1) } \right] 
	= \Exp\left[ 
	e^{-sX} e^{\kappa_A  (\varphi(s) -1) } \right]\,.
	\label{e:varp}
\end{align}
When $\Exp [\kappa_A ]=\kappa<1$,
it is known that this functional equation has  a unique  solution $\varphi$ which further  uniquely determines the distribution of $D$.
By studying the behaviour of the derivatives of $\varphi(s)$ for $s\to 0+$, we get the following result.
\begin{lemma} \label{lm:1}
	If $\Exp X^2 <\infty$ and $\Exp \kappa_A^2<\infty$ then 
	\[
	\Exp D^2=\frac{\Exp {X^2}}{1-\kappa}+\frac{(\Exp{X})^2}{(1-\kappa)^3}\Exp \kappa_A ^2+2\frac{\Exp X}{(1-\kappa)^2}\Exp (X\kappa_A )<\infty\,.
	\]
\end{lemma}
Notice that this expression  coincides with the expression in \cite{zhu}, when $X=f(A)\equiv 1$,
i.e. in the case when one simply counts the number of claims.
\begin{proof}
	Differentiating the equation \eqref{e:varp}   with respect to $s>0$ produces
\[
\varphi'(s) = 
\Exp \left[ 
e^{-sX} e^{\kappa_A  (\varphi(s) -1) }  
\left(  -X + \kappa_A  \varphi' (s) \right)\right]\,.
\]
As $\Exp(\kappa_A)=\kappa<1$ we obtain
\begin{equation}\label{eq:idl1}
\varphi'(s) = \frac{-\Exp \left[ 
	e^{-sX} e^{\mu (\varphi(s) -1) }  
	X  \right]}{1-\Exp\left[ \kappa_A ^{-sX} e^{\kappa_A  (\varphi(s) -1) }\kappa_A \right] }
\end{equation}	
	As $\varphi(s)\le 1$, $s\ge 0$, the  integrand in the numerator is dominated by $X$ and the one in the denominator by $\kappa_A $. By the dominated convergence argument, $\lim_{s\to 0+}\varphi'(s)$ exists and is equal to
	\[
	\varphi'(0) = \dfrac{- \Exp X}{1-\kappa}\,.
	\] 
	In particular $\Exp D = \Exp X/(1-\kappa)$. 
	Differentiating \eqref{e:varp}  again produces second moment of $D$. Indeed, we have
	\[
	\varphi''(s) = 
	\Exp \left[ 
	e^{-sX} e^{\kappa_A  (\varphi(s) -1) }  
	\left(\left(  -X + \kappa_A  \varphi' (s) \right)^2 + \kappa_A  \varphi'' (s) \right)\right]
	\,,
	\]
	so that
\begin{equation}\label{eq:idl2}
	\varphi''(s) = 
\frac{\Exp \left[ 
	e^{-sX} e^{\kappa_A  (\varphi(s) -1) }  
	\left(  -X + \kappa_A  \varphi' (s) \right)^2\right]}{1-\Exp[e^{-sX} e^{\kappa_A  (\varphi(s) -1) }\kappa_A ]}.
\end{equation}

	Here again, applying the dominated convergence theorem twice, one can let $s\to 0+$ and obtain
	\[
	\varphi''(0) = \dfrac{\Exp \left(  -X + \kappa_A  \varphi' (0) \right)^2}{1-\kappa}=\dfrac{\Exp \left(  X + \kappa_A  \Exp D \right)^2}{1-\kappa}\,.
	\]
	Which concludes the proof since  $X=f(A)$.
\end{proof}

The following theorem describes the behavior of the total claim amount $(S(t))$ for the marked Hawkes process under appropriate 2nd moment assumptions. Recall from \eqref{e:muA} that 
$
\tilde{\mu}_{A} (B) = \int_B h(s, A) ds\,.
$

\begin{thm} \label{thm:CLTHawkes}
	If $\kappa<1$,  $\Exp X^2 <\infty$ and $\Exp[ \kappa_A ^2]<\infty$ then, in either stationary or nonstationary case, if
	\begin{align}\label{eq:condh3}
		\sqrt t  \Exp[\tilde{\mu}_A(t,\infty)]\to 0,\qquad t\to \infty,
	\end{align}
then the relation 
\eqref{eq:CLT} holds.
\end{thm}
\begin{proof}
In order to apply Proposition \ref{prop:CLT} one has to check that $\vep_t = o_P(\sqrt t)$.
The proof is based on the following  domination argument on $\vep_t$. Recall that one can write
  $$
  N = \dsum_i \dsum_j \delta_{\taui + T_{ij},A_{ij}} = \dsum_{k=1}^\infty \delta_{\tau_k,A^{k}} \,,
  $$
  w.l.o.g. assuming that $0\leq \tau_1\leq \tau_2\leq \ldots$.
  At each time $\tau_j$, a claim arrives  generated by one of the previous claims or
  an entirely new (immigrant) claim  appears.
  In the former case, if $\tau_j$ is a direct offspring of a claim at time $\tau_i$, we will write
  $\tau_i\to \tau_j$. Progeny $\tau_j$ then creates potentially further  claims. We denote by $D_{\tau_j}$ the total amount of claims generated by the arrival at $\tau_j$ (counting the claim at $\tau_j$ itself as well). Clearly, $D_{\tau_j}$'s are identically distributed as $D$ and even independent 
  if we consider claims which are not offspring of one another. They are also 
  independent of everything happening in the past.

The process $N$ is naturally dominated by
   the stationary marked Hawkes process $N^*$ which is well defined on the whole real line as we assumed  
   $\kappa = \Exp \kappa_A  <1$, see discussion at the end of Subsection~\ref{ss:Mod}. For the original and stationary Hawkes processes, $N$ and $N^*$, by
    $\lambda$ and $ \lambda^*$, we denote corresponding predictable intensities. By the construction of these two point processes,
    $\lambda \leq \lambda^*$. Recall that 
   $\tau_i\to \tau_j$ is equivalent to $\tau_j=\tau_i+W_{ik}$, $k\le L^i = L_{A^i}$, where, by assumption, $W_{ik}$ are i.i.d. with common density $h(t,A^i)/ \kappa_{A^i}$, $t\ge0$, and independent of $L^{i}$ conditionally on the mark $A^i$ of the claim at $\tau_i$. Moreover,
   conditionally on $A^i$, the number of direct progeny of the claim at $\tau_i$, denoted  by $L^{i}$, has Poisson distribution with parameter $\tilde{\mu}_{A^i}$. Therefore, 
using conditional independence and equal distribution of $D^{'}$s we get
  \begin{align*}
\Exp[\vep_t] &=  \Exp\Big[ \dsum_{\taui \leq t}\dsum_j \1_{\taui + T_{ij} >t}  X_{ij}\Big]\\
   &= \Exp\Big[\dsum_{\tau_i \leq t} \dsum_{\tau_j > t}  D_{\tau_j}  \1_{\tau_i\to \tau_j}\Big]\\
   &=\Exp\Big[\dsum_{\tau_i \leq t} \EE\Big[\dsum_{k=1}^{L^{i}} D_{\tau_i+W_{ik}}  \1_{\tau_i+W_{ik} > t}  \mid (\tau_i,A^i)_{i\ge 0}; \tau_i \le t\Big]\Big]\\
   &= \mu_D  
   \Exp  \left[ \dint_{0}^t \dint_{\mathbb{S}} \tilde{\mu}_a((t-s,\infty)) N(ds,da)  \right]\,,
\end{align*}
where $ \tilde{\mu}_{a} ((u,\infty)) = \int_u^\infty h(s, a) ds$.
Observe that  from 
projection theorem, see \cite{brem}, Chapter 8, Theorem 3, the last expression equals to
\[
 \mu_D \EE\left[\dint_0^t \dint_{\mathbb{S}} \tilde{\mu}_a((t-s,\infty)) Q(da) \lambda(s)ds   \right]\,,
\]

One can further bound this estimate by
\begin{align*}
\EE\left[\dint_0^t \dint_{\mathbb{S}} \tilde{\mu}_a((t-s,\infty)) Q(da) \lambda^*(s)ds\right]&=\dint_0^t \dint_{\mathbb{S}} \tilde{\mu}_a((t-s,\infty)) Q(da) \EE[\lambda^*(s)]ds\\
&=\frac{\nu}{1-\kappa}\dint_0^t \dint_{\mathbb{S}} \tilde{\mu}_a((t-s,\infty)) Q(da) ds
\end{align*}

Here we used Fubini's theorem, and the  expression   $\Exp  \left[ \lambda^*(s)   \right] \equiv \nu /(1-\kappa)$. Observe that this expectation is  constant  since $N^*$ is a stationary point process, to show that it equals $ \nu/(1-\kappa)$, note that


	\begin{eqnarray*}
		\lefteqn{ {\mu}^* = \Exp { \lambda^*(s)}
			= \Exp \left[ \nu + \dint_{-\infty}^s \dint_{\mathbb{S}} h(s-u,a) N^*(du,da) \right]}\\
		& = & \nu + \dint_{-\infty}^s \dint_{\mathbb{S}} h(s-u,a) \Exp({\lambda}^*(u))duQ(da)\\ 
		& = & \nu + {\mu}^* \dint_{-\infty}^s \EE h(s-u,A)du \\ 
		& = & \nu + {\mu}^*\dint_0^{\infty}\Exp h(v,A)dv\,, 
	\end{eqnarray*} 
	see also \cite{daley}, Example 6.3(c).
	Hence, ${\mu}^* = \nu + {\mu}^* \cdot \kappa$ and $ {\mu}^* = {\nu}/({1-\kappa})$  as we claimed above.
	Now, we have 
	\begin{equation}\label{eq:domH}
	\Exp {\vep_t} \leq \dfrac{\nu  }{1-\kappa}
	\dint_0^t \dint_{\mathbb{S}} \tilde{\mu}_a((t-s,\infty)) Q(da)ds
	= \dfrac{\nu  }{1-\kappa}
	\dint_0^t \mu_D \int_s^\infty \Exp[h(u,A)]du  ds \,.
	\end{equation}
 Hence the residual term is bounded in expectation by the 
 expression we obtained in the mixed binomial case in Section \ref{sec:mixbin}. Thus, the result will follow from the proof of Corollary \ref{cor:CLTmix} under the condition \eqref{eq:condCLTmix} which is further equivalent to  \eqref{eq:condh3} thanks to the expression of the density of the waiting times.

	Dividing the last expression  by $\sqrt t$ and applying L'H\^opital's rule,  proves the theorem for the nonstationary or pure Hawkes
	process, see \cite{zhu} where the same idea appears in the proof of Theorem 1.3.2. 
	
	To show that the the central limit theorem holds in the stationary case, note that   $S(t)$ now has a similar representation as in \eqref{e:Soft} but with an additional term on the right hand side, i.e.
	\begin{equation} \label{e:SoftStat}
	S(t) = \sum_{i=1}^{\tau(t)} D_i 
	- D_{\tau(t)}  - \vep_t +\vep_{0,t}\,,\ t \geq 0\,,
	\end{equation}
	where 
	$$
	\vep_{0,t} = \dsum_{\Gamma_i \leq 0 ,\  0 < \Gamma_i+T_{ij}<t}  X_{ij} \, .
	$$

	Similar computation provides 
	\begin{eqnarray*}
		\lefteqn{ \Exp {\vep_{0,t}}= \Exp \dsum_{\Gamma_i \leq 0}
			\dsum_j \1_{0<\Gamma_i + T_{ij} <t} X_{ij}
			= \Exp \dsum_{\tau_i \leq 0} \dsum_{0< \tau_j<t}  D_{\tau_j}  \1_{\tau_i\to \tau_j}}\\
		&= & 
		\Exp \left[  \dsum_{\tau_i \leq 0} \mu_D
		\Exp \left(  \dsum_{0<\tau_j<t}  \1_{\tau_i\to \tau_j}
		\bigg|\, \mathcal{F}_0 \right)
		\right]   \\
		& = & \mu_D
		\Exp  \left[  \dsum_{\tau_i \leq 0}   \tilde{\mu}_{A^{i}}((0-\tau_i,t-\tau_i)) \right]   \\ 
		& = & \mu_D  
		\Exp  \left[ \dint_{-\infty}^0 \dint_{\mathbb{S}} \tilde{\mu}_a((-s,t-s)) N^*(ds,da)  \right].  
	\end{eqnarray*} 
	where we denote $
	\tilde{\mu}_{a} (B) = \int_B h(s, a) ds\,$
and $\mathcal{F}_0$ stands for the internal history of the process up to time $0,$ i.e. $\mathcal{F}_0 = \sigma\{ N(I \times S): I \in \mathcal{B}(\mathbb{R}), I \subset (-\infty, 0], S \in \mathcal{S}\}.$ 
Again, by the projection theorem, see \cite{brem}, Chapter 8, Theorem 3, the last expression equals to
\begin{eqnarray*}	
		\mu_D
		\Exp  \left[ \dint_{-\infty}^0 \dint_{\mathbb{S}} \tilde{\mu}_a((-s,t-s)) \lambda^*(s)dsQ(da)  \right].
	\end{eqnarray*} 
Which is further equal to
\begin{eqnarray*}		
		\lefteqn{\mu_D
		\dint_{-\infty}^0 \dint_{\mathbb{S}} \tilde{\mu}_a((-s,t-s)) \Exp \left[   \lambda^*(s) \right] dsQ(da)} \\
		& = & 
		\mu_D \dfrac{\nu}{1-\kappa} 
		\dint_{-\infty}^0 \dint_{\mathbb{S}} \tilde{\mu}_a((-s,t-s)) dsQ(da) \\
		& = & 
		\mu_D \dfrac{\nu}{1-\kappa} 
		\dint_{-\infty}^0 \Exp \tilde{\mu}_A((-s,t-s))  ds \\
		& = & 
		\mu_D \dfrac{\nu}{1-\kappa} 
		\dint_{0}^{\infty} \Exp \tilde{\mu}_A((s,s+t))  ds \\
		& = &  \mu_D \dfrac{\nu}{1-\kappa} 
		\dint_0^{\infty} \Exp \dint_s^{s+t} h(u,A)du   ds\\
		& = & 
		\mu_D\dfrac{\nu}{1-\kappa} 
		\dint_0^{\infty} \Exp (t\wedge u) h(u,A) du\\
		&=&  \mu_D\dfrac{\nu}{1-\kappa} 
		\left( \dint_0^{t} \Exp [ u h(u,A) ] du+t\dint_t^{\infty} \Exp [ h(u,A) ] du \right).
	\end{eqnarray*} 
	Notice that the second term in the last expression divided by $\sqrt t$ tends to $0$ by \eqref{eq:condh3}. 
	Using integration by parts for the first term, we have
	$$
	\dint_0^{t} \Exp [ u h(u,A) ] du = t \dint_t^\infty \Exp[h(s,A)]ds+\dint_0^t\dint_u^\infty \Exp[h(s,A)]dsdu.
	$$
	The first integral on the r.h.s. divided by $\sqrt t$ tends to $0$ under \eqref{eq:condh3}. The last term divided by $\sqrt t$ also tends to $0$ by an application of the L'H\^opital rule 
	as in the non-stationary case.
	
	Finally, we observe that $\vep_{0,t}/\sqrt t \cip 0$ and the result in the stationary case is proved.
\end{proof}

	Observe that \eqref{eq:condh3} is substantially weaker than \eqref{eq:condstat} in the unmarked case. Namely the former condition only requires that the total residue due to the claims on the compact interval $[0,t]$ is of the order $o(\sqrt t)$ in probability.  In particular, in the unmarked case, the central limit theorem holds for the stationary and the non-stationary case even if \eqref{eq:condstat} is not satisfied, i.e. even when non-stationary process is not convergent. 

As we mentioned above,  there are  related limit theorems in the literature concerning only the counting process $N_t$, see \cite{zhu}, but in the contrast to their result,  our proof
does not rely on the martingale central limit theorem, it stems from  rather simple relations 
\eqref{e:Soft} and  \eqref{e:SoftStat}.

In the following example, we consider some special cases of Hawkes processes for which a closed form expression for the 
2nd moment $\Exp D^2$ can be found.

\begin{example} 
	
({Marked Hawkes processes with claims independent of the cluster size})
Assume  that the random measure \eqref{e:muA}
$$ 
\tilde{\mu}_A (B  ) = \dint_B h(s, A) ds\,,
$$
on $\mathbb{R}_+$ and
the  corresponding claim size $X=f(A)$ are independent. 
In particular, this holds if 
$\tilde{\mu}_A (B) = \int_B h(s) ds\,,$ for some integrable function $h$, i.e. when $\tilde{\mu}_A$ is a deterministic measure and we actually have standard  Hawkes process with independent marks.
In this special case $K+1$ is known to have the so--called Borel distribution, see \cite{borel}.

{
	Using the arguments from the proof of Lemma~\ref{lm:1},
one obtains
$
\mu_D = \Exp D_i = {\Exp X}/({1-\kappa}).
$
Similarly the variance of $D_i$'s is finite as the variance of a compound sum, and equals
$$
\sigma_D^2 =  \dfrac{{\sigma_X^2}}{1-\kappa}  + \dfrac{\kappa {{(\Exp X)}^2} }{(1-\kappa)^3}\,,
$$ 
cf. Lemma 2.3.4 in  \cite{mikosch}.
 Hence $\Exp D^2 $ in 
Theorem~{\ref{thm:CLTHawkes}} has the form
$$
\Exp D^2 = \sigma_D^2 + \mu_D^2  = {\frac{\sigma_X^2}{1-\kappa} + \frac{{(\Exp X)}^2}{(1-\kappa)^3}}.
$$
In the special case, when the claims are all constant, say $X=f(A)\equiv c>0$, direct calculation yields
$\Exp D = c/({1-\kappa})\,,$
with $\kappa=\Exp [\kappa_A ]$,
 and
\[
\varphi''(0) = \dfrac{\Exp \left(  -c+ \kappa_A  \varphi' (0) \right)^2}{1-\kappa} = {c^2}\dfrac{\Var \kappa_A  + 1}{(1-\kappa)^3}\,,
\]
obtaining 
\[
\Exp D^2 = \varphi''(0) ={c^2} \dfrac{\Var \kappa_A  + 1}{(1-\kappa)^3}\,,
\]
in particular, for $c=1$ we recover expression in \cite{zhu}.}

\end{example}

In the rest of this  subsection, we study marked Hawkes process in the case when $D_i$'s are regularly varying
with index $\alpha <2 $. Using the result of \cite{HuSa}, one can show that when the individual claims $X=f(A)$ are regularly varying, this property is frequently 
passed on to the random variable $D$ under appropriate moment assumptions on $\kappa_A $. {However, using the specific form of the Laplace transform for $D$ given in \eqref{e:varp}, one can show regular variation of $D$ under weaker conditions}. This is the content
of the following lemma.

\begin{lemma} \label{lem:RVHawkes}
	Assume that $\kappa<1$ and that $X=f(A)$ is
	regularly varying with index $\alpha \in (0,1)\cup (1,2)$. When $\alpha \in (1,2)$, {assume additionally that $Y=X + \kappa_A \mu_D$ is regularly varying of order $\alpha$. }
	Then the random variable $D$ 
	is regularly varying with the same index $\alpha$\,.
\end{lemma}

\begin{proof} 
	
		We will use Karamata's Tauberian Theorem, as formulated and proved in Theorem 8.1.6 of~\cite{bingham}. In particular, the equivalence between (8.1.12) and (8.1.11b) in 
		\cite{bingham} yields the following.
	\begin{thm}\label{th:eqrv}
			The nonnegative random variable $X$ is regularly varying with a noninteger tail
			index $\alpha>0$, i.e. $\bar F(x)\sim x^{-\alpha}\ell(x)$ as $x\to\infty$ if and only if 
			\[
			\varphi^{(\lceil \alpha \rceil)}(s) \sim cs^{\alpha-\lceil \alpha \rceil}\ell(1/s),\qquad s\to 0+,
			\]
			for some slowly varying function $\ell$ and a constant depending only on $\alpha$: $c=-\Gamma(\alpha+1)\Gamma(1-\alpha)/\Gamma(\alpha-\lfloor \alpha \rfloor)$.
	\end{thm}

		Consider first the case $0<\alpha<1$. By differentiating once the expression for the Laplace transform, we obtain the identity \eqref{eq:idl1}
		\[
		\varphi'(s) =\frac{-\Exp \left[ 
			e^{-sX} e^{\kappa_A  (\varphi(s) -1) }  
			X  \right]}{1-\Exp[e^{-sX} e^{\kappa_A  (\varphi(s) -1) }\kappa_A ] },\qquad s>0.
		\]
		We are interested in the behavior of this derivative as $s\to 0+$. Using the inequality $|1-e^{-x}|=1-e^{-x}\le  x$,  we have 
		\begin{align*}
		\left|\varphi'(s)-\frac{-\Exp \left[ 
			e^{-sX}   
			X  \right]}{1-\Exp[e^{-sX-\kappa_A (1-\varphi(s))}\kappa_A ] }\right|&\le  \frac{\Exp \left[  \kappa_A  (1-\varphi(s))
			e^{-sX}   
			X  \right]}{1-\Exp[e^{-sX-\kappa_A (1-\varphi(s))}\kappa_A ] }\\
		&\hspace{-2cm}\le  \frac{1-\varphi(s) }s\frac{\Exp \left[\kappa_A  
			e^{-sX}   
			sX  \right]}{1-\Exp[e^{-sX-\kappa_A (1-\varphi(s))}\kappa_A ] }.
		\end{align*}
		As $e^{-sX} sX\le e^{-1},$ we prove that $\Exp \left[\kappa_A  
		e^{-sX}   
		sX  \right]=o(1)$ as $s\to 0+$ by dominated convergence. As in the proof of Lemma \ref{lm:1} the denominator $1-\Exp[e^{-sX-\kappa_A (1-\varphi(s))}\kappa_A ]$ is controlled thanks to dominated convergence as well.
		 Moreover, using again $1-e^{-x}\le  x$ and denoting $\varphi_X(s)=\Exp[e^{-sX}]$ the Laplace transform of $X$,  we have
		\[
		0\le \varphi_X(s)-\varphi(s)\le \Exp[e^{-sX}\kappa_A (1-\varphi(s)) ]\le \kappa (1-\varphi(s))
		\]
		so that
		\[
		1-\varphi(s)\le \frac1{1-\kappa}(1-\varphi_X(s)).
		\]
		Collecting all those bounds and using the identity $\varphi'_X(s)=-\Exp[e^{-sX}X]$, we obtain
		\begin{equation}\label{e:eql}
		\left|\varphi'(s)-\frac{\varphi'_X(s)}{1-\Exp[e^{-sX} e^{\kappa_A  (\varphi(s) -1) }\kappa_A ]} \right| =o\left(\frac{1-\varphi_X(s)}s\right),\qquad s\to 0^+.\end{equation}
		The regular variation of the random variable $D$ follows now from the regular variation of the random variable $X$ by two consecutive applications of Theorem \ref{th:eqrv}. First, as $X$ is regularly varying of order $0<\alpha<1$, applying the direct part of the equivalence in Theorem \ref{th:eqrv} we obtain
		\[
		\varphi_X'(s)\sim c s^{\alpha -1 }\ell(1/s), \qquad s\to 0^+.
		\]
		Applying Karamata's theorem,  i.e. the equivalence between (8.1.9) and (8.1.11b) in 
		\cite[Theorem 8.1.6]{bingham}, we obtain
		 $(1-\varphi_X(s))/s=O(\varphi'_X(s))$ as $s\to 0+$.
		Using \eqref{e:eql} and the limiting relation 
		$$\Exp[e^{-sX} e^{\kappa_A  (\varphi(s) -1) }\kappa_A ]\to \kappa,\qquad s\to 0^+\,,
		$$ we obtain
		\[
		\varphi'(s)\sim \frac{\varphi'_X(s)}{1-\kappa}\sim \frac{c s^{\alpha -1 }\ell(1/s)}{1-\kappa},\qquad s\to 0^+.
		\]
		Finally, applying the reverse part of Theorem \ref{th:eqrv}, we obtain
		\[
		\bar F_D( x)\sim \frac{\ell(x)x^{-\alpha   }}{1-\kappa}=\frac{\bar F_X(x)}{1-\kappa},\qquad x\to \infty.
		\]	
		The case $1<\alpha<2$ can be treated similarly, under the additional assumption that $Y=X + \kappa_A  \mu_D$ is regularly varying. We will again show  that $P(D>x)\sim (1-\kappa)^{-1}P(Y>x)$ as $x\to \infty$. To prove this equivalence, recall the identity \eqref{eq:idl2}
		\[
		\varphi''(s) = 
		\frac{\Exp \left[ 
			e^{-sX} e^{\kappa_A  (\varphi(s) -1) }  
			\left(  -X + \kappa_A  \varphi' (s) \right)^2\right]}{1-\Exp[e^{-sX} e^{\kappa_A  (\varphi(s) -1) }\kappa_A ]}.
		\]
		As $\alpha>1$, we have that $\Exp[Y]<\infty$ and thus $ \Exp[X]<\infty$ and $E[D]=\mu_D=(1-\kappa)^{-1}\Exp[X]$.
		Observe that, for any $s>0$,
		\begin{align*}
		&\left|\varphi''(s)-\frac{\Exp \left[ 
			e^{-sY}   
			Y^2  \right]}{1-\Exp[e^{-sX-\kappa_A (1-\varphi(s))}\kappa_A ] }\right|\\&\hspace{1cm}=\left|\frac{\Exp\left[e^{-sX-\kappa_A  (1-\varphi(s) ) }  
			\left(  -X + \kappa_A  \varphi' (s) \right)^2\right]-\Exp \left[ 
			e^{-sY}   
			Y^2  \right]}{1-\Exp[e^{-sX-\kappa_A (1-\varphi(s))}\kappa_A ] }\right|		.	
		\end{align*}
		Let us decompose the numerator into two terms
		\begin{align*}
		&\underbrace{\Exp\left|e^{-sX-\kappa_A  (1-\varphi(s))}  \left(
			\left(  -X + \kappa_A  \varphi' (s) \right)^2-Y^2\right)\right|}_{I_1}\\
		&+\underbrace{\Exp\left|\left(
			e^{-sY}-e^{-sX-\kappa_A  (1-\varphi(s))} \right)
			Y^2  \right|}_{I_2} \,.
		\end{align*}
		Using the identity $a^2-b^2=(a-b)(a+b)$, $I_1$ is bounded by
		\begin{align*}
		I_1&\le
		(\mu_D + \varphi'(s))\Exp\left[e^{-sX-\kappa_A  (1-\varphi(s))}\kappa_A \left(2X+\kappa_A (\mu_D-\varphi'(s))\right)\right]\\
		&\le \frac{\mu_D + \varphi'(s)}s\left(\Exp\left[2\kappa_A e^{-sX}sX\right]+\Exp\left[\kappa_A e^{-\kappa_A  (1-\varphi(s))}s(\mu_D-\varphi'(s))\right]\right).
		\end{align*}
		As $e^{-sX}sX\le e^{-1}$ then $\Exp\left[2\kappa_A e^{-sX}sX\right]=o(1)$ as $s\to 0^+$ by dominated convergence. By convexity of $\varphi(s)$ we have $1-\varphi(s)\ge -\varphi'(s)s$ for any $s>0$. Thus 
		\[
		e^{-\kappa_A  (1-\varphi(s))}(-\varphi'(s)s)\le e^{-\kappa_A  (-\varphi'(s)s)}(-\varphi'(s)s)\le e^{-1}
		\]
		and the dominated convergence argument also applies to the second integrand as $-\varphi'(s)s\le 1-\varphi(s)=o(1)$. We obtain $I_1=o((\mu_D + \varphi'(s))/s)$ as $s\to 0^+$. In order to control the rate of $(\mu_D + \varphi'(s))/s$, we notice that $\varphi(s)$ is $\mu_D$ Lipschitz on $s\ge0$ so that $|1-\varphi(s)|=1-\varphi(s)\le \mu_D s$. Then
		\[
		sX+\kappa_A  (1-\varphi(s))\le sX+s\kappa_A \mu_D =sY
		\]
		and we bound
		\begin{align*}
		\varphi'(s)&\le \frac{-\Exp\left[e^{-sY}X\right]}{1-\Exp[e^{-sY }\kappa_A ]}\\
		&\le \frac{\varphi_Y'(s)}{1-\Exp[e^{-sY }\kappa_A ]}+\frac{\kappa\mu_D}{1-\Exp[e^{-sY }\kappa_A ]}
		\end{align*}
		where $\varphi_Y(s)=\Exp[e^{-sY}]$ denotes the Laplace transform of $Y$. It yields to the estimates $\mu_D + \varphi'(s)=O(\mu_D + \varphi_Y'(s))+ O(\kappa-\Exp[e^{-sY }\kappa_A ])$. Using again that $1-e^{-x}\le x$ on the second term we obtain that $I_1=o((\mu_D + \varphi_Y'(s))/s)+o(1)$ as $s\to 0^+$.

		We now turn to the term $I_2$  that we identify as
		\begin{align*}
		I_2&= \Exp \left|
		\left(e^{sX+\kappa_A  (1-\varphi(s))-sY}-1\right) e^{-sX-\kappa_A  (1-\varphi(s))}
		Y^2  \right|\\
		& = \Exp \left|
		\left(e^{-\kappa_A (\mu_D s- (1-\varphi(s)))}-1\right) e^{-sX-\kappa_A  (1-\varphi(s))}
		Y^2  \right| .
		\end{align*}
		As $1-\varphi(s)\le \mu_D s$ the term in the absolute value is negative for $s>0$ and
		\[
		I_2= \Exp \left[
		\left(1-e^{-\kappa_A (\mu_D s- (1-\varphi(s)))}\right) e^{-sX-\kappa_A  (1-\varphi(s))}
		Y^2  \right] .
		\]
		Using again the basic inequality $1-e^{-x}\le  x$ for $x\ge 0 $ we obtain the new estimate
		\begin{align*} 
		I_2&\le   \Exp \left[
		\kappa_A  (s\mu_D-(1-\varphi(s))
		e^{-sX-\kappa_A  (1-\varphi(s))}Y^2  \right]\\
		&\le  \frac{s\mu_D-(1-\varphi(s))}{s^2}\Exp \left[
		\kappa_A  
		e^{-sX-\kappa_A  (1-\varphi(s))}(sY)^2  \right].
		\end{align*}
		We have
		\[
		(sY)^2\le (sX +\kappa_A  (1-\varphi(s)))^2 + \kappa_A^2(s\mu_D-(1-\varphi(s)))^2\,.
		\]
		As $e^{-x} x^2\le 4e^{-2}$ for any $x>0$, we prove that
		\[
		\Exp \left[\kappa_A  
		e^{-(sX+\kappa_A  (1-\varphi(s))}(sX+\kappa_A  (1-\varphi(s)))^2\right]=o(1)
		\]
		as $s\to 0^+$ by dominated convergence. It remains to bound the term 
		\[
		e^{-sX-\kappa_A  (1-\varphi(s))} \kappa_A^2(s\mu_D-(1-\varphi(s)))^2
		\]
		uniformly for $s>0$ sufficiently small. As $1-\varphi(s)\sim s\mu_D$ as $s\to 0^+$, we have $0\le s\mu_D-(1-\varphi(s))\le 1-\varphi(s)$ for $s$ sufficiently small. Then we obtain 
		\[
		e^{-sX-\kappa_A  (1-\varphi(s))} \kappa_A^2(s\mu_D-(1-\varphi(s)))^2\le e^{-\kappa_A  (1-\varphi(s))}\kappa_A^2(1-\varphi(s))^2=o(1)
		\]
		where the negligibility follows from dominated convergence and the basic inequality $e^{-x} x^2\le 4e^{-2}$ for any $x>0$.
		We obtain
		\[
		I_2=o\left(\frac{s\mu_D-(1-\varphi(s))}{s^2}\right),\qquad s\to 0^+.
		\]
		Similar computation than above yields
		\begin{multline*}
		0\le \varphi(s)-\varphi_Y(s)\le \Exp \left[
		\kappa_A  (s\mu_D-(1-\varphi(s))
		e^{-sX+\kappa_A (\varphi(s)-1))}\right]\\\le \kappa(s\mu_D-(1-\varphi(s)).
	\end{multline*} 
Thus as	 $(s\mu_D-(1-\varphi(s))\le (s\mu_D-(1-\varphi_Y(s))/(1-\kappa)$ and 
 $\Exp[Y]=\Exp[X]+\kappa \mu_D=\mu_D$ we conclude that
		\begin{multline*}
		\left|\varphi''(s)-\frac{\varphi_Y''(s)}{1-\Exp[e^{-sX-\kappa_A (1-\varphi(s))}\kappa_A ] }\right|\\
		=o\left(1+\frac{\Exp[Y] + \varphi_Y'(s)}s+\frac{s\Exp[Y]-(1-\varphi_Y(s))}{s^2}\right),
		\end{multline*}
		as $s\to 0^+$.
		Let us first apply Theorem \ref{th:eqrv} on $Y$ so that $\varphi''_Y(s)$ is $\alpha-2$ regularly varying around $0$. Applying Karamata's theorem again, i.e the equivalences between  (8.1.11b) and (8.1.9), (8.1.11b)  and (8.1.10) in 
		\cite[Theorem 8.1.6]{bingham}  assert respectively that $(s\Exp[Y]-(1-\varphi_Y(s)))/s^2=O(\varphi_Y''(s))$ and $(\Exp[Y] + \varphi_Y'(s))/s=O(\varphi_Y''(s))$ as $s\to 0^+$. We then obtain
		\[
		\varphi''(s)\sim \frac{\varphi''_Y(s)}{1-\kappa}\sim \frac{c s^{\alpha -2 }\ell(1/s)}{1-\kappa},\qquad s\to 0+,
		\]
and finally $\bar F(x) \sim \bar F_Y(x)/(1-\kappa)$, $x\to \infty$, by applying the reverse part of Theorem \ref{th:eqrv}.
\end{proof} 
{We are now ready to characterize the asymptotic behavior of $S(t)$ in the regularly varying case.}

\begin{thm} \label{thm:RVHawkes}
	Assume that the assumptions of Lemma~\ref{lem:RVHawkes} hold.\\
		i) If $\alpha \in (0,1)$ and there exists $\delta>0$ such that 
		\begin{equation}\label{e:uhu}
		t ^{ \delta  }\EE [\tilde{\mu}_A(t,\infty)] \to 0.
		\end{equation}	
	as $t \toi$,  then there exists a sequence $(a_n),\, a_n\toi ,$ and an $\alpha$--stable random variable $G_\alpha$ such that
	 	$$
	 	\frac{S(t) }{ 
	 		a_{\lfloor \nu t \rfloor} }
	 	\cid G_\alpha\,.
	 	$$
		ii) If $\alpha \in (1,2)$ and

	\begin{equation}\label{e:vepdelt}
	t ^{1+\delta - 1/\alpha} 
   \EE [\tilde{\mu}_A(t,\infty)]    \to 0\,,
	\end{equation}	
	as $t \toi$  holds for some  $\delta >0$, then there exists a sequence  $(a_n),\, a_n\toi ,$ and an $\alpha$--stable random variable $G_\alpha$ such that
		$$
		\frac{S(t) - t \nu \mu_D }{ 
			a_{\lfloor \nu t \rfloor} }
		\cid G_\alpha\,.
		$$
\end{thm}

\begin{proof}
	The proof is based on  the representation
	\eqref{e:Soft},
	and  application of Propositions~\ref{prop:12} and \ref{prop:01}. In either case, it remains to show that
	$$
	 \vep_t = o_P(a_t)\,.
	$$
	Consider first the case $\alpha \in (1,2)$. Since then 
	$\mu_D = \Exp D<\infty$, the argument in the proof of Theorem~\ref{thm:CLTHawkes} still yields the bound
	\eqref{eq:domH} on $\Exp \vep_t$. Using L'H\^opital's rule again together with condition 
	\eqref{e:vepdelt}, shows that
	$\Exp \vep_t = o (t^{1/\alpha-\delta})$,
	where we assume without loss of generality that $\delta<1/\alpha$. Since, $a_t = t^{1/\alpha} \ell(t)$
	for some slowly varying function $\ell$, it follows that
	$\vep_t /a_t\cip 0$ as $t\toi$.

	For $\alpha \in (0,1)$, random variable $D$ has no finite mean. In order to prove that $
	 \vep_t = o_P(a_t)
	$ we use the Markov inequality of order $0<\gamma<\alpha$ as $\Exp [D^\gamma]<\infty$. We will show that under assumption \eqref{e:uhu}
	\[
	\Exp [\vep_t^\gamma]=o(a_t^\gamma)\,,\qquad t\to \infty.
	\]
By sub-linearity of $x\to x^\gamma$, $\gamma\le1$, we have 
\begin{align*}
\Exp[\vep_t^\gamma] &=  \Exp\Big[\Big( \dsum_{\taui \leq t}\dsum_j \1_{\taui + T_{ij} >t}  X_{ij}\Big)^\gamma\Big]\\
   &= \Exp\Big[\Big(\dsum_{\tau_i \leq t} \dsum_{\tau_j > t}  D_{\tau_j}  \1_{\tau_i\to \tau_j}\Big)^\gamma\Big]\\
      &\leq \Exp\Big[ \dsum_{\tau_i \leq t} \dsum_{\tau_j > t}  D_{\tau_j}^\gamma  \1_{\tau_i\to \tau_j} \Big]\\
   &=\Exp\Big[\dsum_{\tau_i \leq t} \EE\Big[\dsum_{k=1}^{L^{i}} D_{\tau_i+W_{ik}}^\gamma  \1_{\tau_i+W_{ik} > t}  \mid (\tau_i,A^i)_{i\ge 0}; \tau_i \le t\Big]\Big]\\
   &= \Exp[D^\gamma] 
   \Exp  \left[ \dint_{0}^t \dint_{\mathbb{S}} \tilde{\mu}_a((t-s,\infty)) N(ds,da)  \right]\,,
\end{align*}
	We can again compare the marked Hawkes process $N$ with a stationary version of it,  $N^*$ say. By the same arguments as in the proof of Theorem \ref{thm:CLTHawkes}, we obtain
		\begin{eqnarray*}
  \Exp \left[ \dsum_{0\le \tau_i \leq t} \dsum_{t< \tau_j}    \1_{\tau_i\to \tau_j} \right]\le  \dfrac{\nu}{1-\kappa}\left(\dint_0^t\dint_u^\infty \Exp[h(s,A)]dsdu\right).
		\end{eqnarray*}
By regular variation of order $1/\alpha$ of $(a_t)$ we have $t^{\gamma/\alpha-\delta'}=o(a_t^\gamma)$ for any $\delta'>0$. Once again, we use a Cesar\'o argument to prove that $\Exp [\vep_t^\gamma]=o(a_t^\gamma)$ under the condition
\[
t ^{1+\delta' - \gamma/\alpha}\EE [\tilde{\mu}_A(t,\infty)] \to 0\,, \qquad t\to \infty\,.
\]
As $\gamma$ can be taken as close as possible to $\alpha$, the result holds under  assumption \eqref{e:uhu}.
\end{proof}

\begin{remark}
Theorem \ref{thm:RVHawkes} i) and ii) also hold on the stationary version following the same arguments as in the proof of Theorem \ref{thm:CLTHawkes}.
\end{remark}

\section*{Acknowledgements}  The work of Bojan Basrak has been supported in part by Croatian
Science Foundation under the project 3526. {The work of Bojan Basrak and Olivier Wintenberger has been supported in part by the AMERISKA network, project ANR-14-CE20-0006.}   The authors thank Philippe Soulier for pointing out the reference \cite{HuSa}.

\bibliographystyle{plainnat}

\end{document}